\documentclass[reqno]{amsart}

\usepackage{amsmath}
\usepackage{amsfonts}
\usepackage{rotating}
\usepackage{bbm}
\usepackage[all,cmtip]{xy}
\usepackage{amscd}

\usepackage{tikz}
\usetikzlibrary{arrows}
\usepackage{subcaption}
 \tikzset{>=latex}
\usetikzlibrary{arrows.meta}

 \usepackage{mathrsfs}
\usepackage{color}
\usepackage{xcolor}
\usepackage{mathtools}
\usepackage{amssymb}
\usepackage[latin1]{inputenc}
\usepackage{graphicx}
\usepackage{dsfont}
\usepackage{color}
\usepackage{hyperref}

\usepackage{marginnote}

\usepackage{amsfonts,amsthm,amssymb,latexsym,amsmath,setspace,amscd,verbatim}

\usepackage{pgfplots}
\pgfplotsset{compat=1.7}

\usepgfplotslibrary{fillbetween}
\usetikzlibrary{patterns}

\newtheorem{theorem}{Theorem}[section]

\newtheorem{lemma}[theorem]{Lemma}
\newtheorem{proposition}[theorem]{Proposition}

\theoremstyle{definition}
\newtheorem{definition}[theorem]{Definition}

\theoremstyle{remark}
\newtheorem{remark}[theorem]{Remark}

\numberwithin{equation}{section}

\newcommand{\p}{\partial}

\newcommand{\D}{{\mathbb{D}}}

\newcommand{\C}{{\mathbb{C}}}

\newcommand{\R}{{\mathbb{R}}}
\newcommand{\Q}{{\mathbb{Q}}}
\newcommand{\Z}{{\mathbb{Z}}}
\newcommand{\N}{{\mathbb{N}}}

 \newcommand{\im}{\mathrm{im}}

\renewcommand{\epsilon}{\varepsilon}
\renewcommand{\theta}{\vartheta}

\usepackage{graphicx}

\title{Transverse foliations in the rotating Kepler problem}

\author[Seongchan Kim]{Seongchan Kim}
  \address{Department of Mathematics Education, Kongju National University, Kongju 32588, Republic of Korea}
  \email {seongchankim@kongju.ac.kr}

\begin{document}
\maketitle

  \begin{abstract}
  We construct  finite energy foliations and  transverse foliations of neighbourhoods of the circular orbits in the rotating Kepler problem for all negative energies. 
  This paper would be a first step towards our ultimate goal that is to  recover and refine McGehee's results on homoclinics \cite{McGehee} and   to establish a theoretical foundation to the numerical demonstration of the existence of a homoclinic-heteroclinic chain  in the planar circular restricted three-body problem \cite{MR1765636}, using pseudoholomorphic curves.

\end{abstract}



  \section{Introduction}
  
     The rotating Kepler problem is the Kepler problem in rotating coordinates, obtained from the planar circular restricted three-body problem (PCR3BP) by setting the mass of one of the primaries to zero.
       Its Hamiltonian  $H \colon T^*( \R^2 \setminus \{  0  \} ) \to \R$ is  given by
   \begin{align}\label{eq:RKPHAM}
   \begin{split}
   H(q_1, q_2, p_1,p_2) &= \frac{1}{2} \lvert p \rvert^2 - \frac{1}{ \lvert q \rvert} - p_2 q_1 + p_1 q_2 \\
   & = \frac{1}{2} ( (p_1 +q_2)^2 + (p_2 - q_1)^2) - \frac{1}{\lvert q\rvert}- \frac{1}{2} \lvert q \rvert^2,
   \end{split}
   \end{align}
where $q=(q_1, q_2)$ and $p=(p_1, p_2).$
 It admits two integrals of motion
   \begin{equation}\label{eq:ELoriginal}
   E = \frac{1}{2} \lvert p \vert^2 - \frac{1}{ \lvert q \rvert }, \quad \quad
   L = q_1 p_2 - q_2 p_1,
   \end{equation}
   called the Kepler energy and the angular momentum, respectively. Throughout the paper, we tacitly assume that the Kepler energy is negative, so that its     trajectories projected  to the $q$-plane are ellipses.   We call such  trajectories Kepler ellipses.

   The Hamiltonian $H$ admits a unique critical value $c = -\frac{3}{2}.$ If $c<-\frac{3}{2},$ then the energy level $H^{-1}(c)$ consists of two connected components, denoted by $\Sigma_c^b$ and $\Sigma_c^u.$  They are called the bounded and the unbounded components, respectively. It is important to mention that the component $\Sigma_c^b$ is not bounded as a set. Indeed, the point $(q_1, q_2, p_1,p_2)=( \frac{1}{a}, 0, \sqrt{ 2(a+c)}, 0)$ lies on $\Sigma_c^b$ for any value of $a>-c.$ We stick to our nomenclature though as its projection to the  $q$-plane  is bounded.  If $c>-\frac{3}{2},$ then the energy level $ H^{-1}(c)$ consists of a single unbounded component whose projection to the $q$-plane   is the whole $\R^2 \setminus \{ 0\}.$

Periodic orbits will be referred to as circular orbits if  the projections to the $q$-plane  are circular. A circular orbit is called direct and retrograde if   it is rotating in the same direction as the coordinate system and   in opposite direction to the coordinate system, respectively.   If $c<-\frac{3}{2},$ then there are precisely two circular orbits $\gamma_{\rm retro}^b$ and $ \gamma_{\rm direct}^b$  on $\Sigma_c^b$ and a unique circular orbit $\gamma_{\rm direct}^u$ on $\Sigma_c^u.$    The circular orbit $\gamma_{\rm retro}^b$ is retrograde and  the other two  are direct. When $c = -\frac{3}{2},$ the two direct circular orbits degenerate into a circle of critical points.
 In the case $c>-\frac{3}{2} ,$ there is a unique circular orbit $\gamma_{\rm retro}^u$ on $H^{-1}(c)$ that is retrograde.

   Fix $c<-\frac{3}{2}.$ 
   The angular momentum $L,$ restricted to $\Sigma_c^b,$ takes values in the interval $[L^b_{\rm retro}, L^b_{\rm direct}],$ where $L^b_{\rm retro}<0$ and    $L^b_{\rm direct}>0$ are angular momenta of  $\gamma_{\rm retro}^b$ and $\gamma_{\rm direct}^b,$ respectively. See Section \ref{sec:RKP}. Each $L \in (L^b_{\rm retro}, L^b_{\rm direct})$ corresponds to a Liouville torus.   The three-dimensional manifold $\Sigma_c^b$ is not compact due to collisions. Note that along all collision orbits we have $L=0.$    However, two-body collisions can always be regularised, see for example \cite[Section 2]{Moser70} and also \cite[Section 4.1]{FvK18book}, and hence $\Sigma_c^b$ can be compactified to form a closed three-dimensional manifold $\overline{\Sigma}_c^b,$  diffeomorphic to $\R P^3.$

   Consider two solid tori 
   \begin{align*}
   \Sigma_{\rm retro} ^b&= \{ (q,p) \in \overline{\Sigma}_c^b : (q,p) \mbox{ lies on an orbit with } L \in [L_{\rm retro}^b,  0  ]\},\\
   \Sigma_{\rm direct}^b &= \{ (q,p) \in \overline{\Sigma}_c^b : (q,p) \mbox{ lies on an orbit with } L \in [0,L_{\rm direct}^b] \}.
   \end{align*}   
  Each solid  torus is foliated into embedded discs, whose boundaries lie on the boundary torus consisting only of collision orbits and whose interiors are transverse to the Hamiltonian flow.    The circular orbit  corresponds to a unique fixed point of the associated first return map.    To see this, we first note that there is a contact form on $\overline{\Sigma}_c^b$ that is dynamically convex, i.e.\ every contractible periodic orbit has Conley-Zehnder index at least three.   See  \cite[Proposition 3.2]{AFvKG12} and   \cite[Theorem 1.1]{RKP13}. Then an application of a result from \cite{HS16elliptic} implies that $\gamma_{\rm retro}^b$   binds a rational open book decomposition whose pages are disc-like global surfaces of section for the Hamiltonian  flow.     Since every $ L \in [0,L_{\rm direct}^b]$ corresponds to a Liouville torus that is tangent to the Hamiltonian flow, each of the pages of this rational open book decomposition intersects $\Sigma_{\rm direct}^b$ transversally in closed discs that determine the above-mentioned disc foliation of $\Sigma_{\rm direct}^b$. The same holds for $\gamma_{\rm direct}^b$ and $\Sigma_{\rm retro}^b.$

   The goal of this paper is to show that   similar assertions hold  in  neighbourhoods of $\gamma_{\rm direct}^u$   and $\gamma_{\rm retro}^u,$   as stated below.

   \begin{theorem} \label{thm:main} The following assertions hold.
\begin{enumerate}
    \item    Assume that $c<-\frac{3}{2}.$ Then  there is a subset  $\Sigma_{\rm direct}^u$  of the unbounded component $\Sigma_c^u \subset H^{-1}(c)$ satisfying the following properties:
   \begin{enumerate}
   \item It is diffeomorphic to a solid torus whose  core is   $\gamma_{\rm direct}^u $ and whose boundary is a Liouville torus.
   \item It is foliated into embedded discs whose  boundaries lie on the boundary torus and whose interiors are transverse to the Hamiltonian flow.   
   \item A unique fixed point of the associated first return map corresponds to  the direct circular orbit $\gamma_{\rm direct}^u.$ 
   \end{enumerate}
   \item   If $-\frac{3}{2}<c<0,$ then there is a subset  $\Sigma_{\rm retro}^u \subset H^{-1}(c)$ having the same properties as above, with  $\gamma_{\rm direct}^u$ being replaced by $\gamma_{\rm retro}^u.$ 
      \end{enumerate}
      \end{theorem}

   More precise descriptions for the subsets $\Sigma_{\rm direct}^u$ and $\Sigma_{\rm retro}^u$ in the statement above  will be given in Sections \ref{sec:lower} and \ref{sec:higher}, respectively.

\begin{remark} Here are a couple of remarks on the theorem above. 
   \begin{enumerate}
       \item   The reason why we consider only negative energies is that we would not like to regularise the collisions. See Section \ref{sec:RKP}.

\item The subset $\Sigma^u_{\rm direct}$ can be arbitrarily large. See Remark \ref{rmk:large}.

   
    \item Pick any embedded disc $\mathscr{D}$ described in (b) above.      Since the rotating Kepler problem is completely integrable, the solid torus $\Sigma_{\rm direct}^u$ is foliated by Liouville tori. This Liouville foliation induces a foliation of the disc $\mathscr{D}$ into concentric circles whose centre corresponds to the direct circular orbit $\gamma_{\rm direct}^u$.
The same holds for $\Sigma_{\rm retro}^u $ with $\gamma_{\rm direct}^u$ being replaced by $\gamma_{\rm retro}^u.$

    \end{enumerate}
\end{remark}

   The motivation of this paper comes from results on the  PCR3BP  by McGehee \cite{McGehee} and by Koon, Lo, Marsden and Ross  \cite{MR1765636}. We briefly state below these results for readers' convenience.

The PCR3BP studies the motion of an infinitesimal body  under the gravitational influence of two massive bodies moving along circular orbits around their   center of mass. 
The infinitesimal body, denoted by $C,$ is assumed to move in the plane spanned by the orbits of the two massive bodies, denoted by $S$ and $J.$  
We rescale the total mass of the system to one, so that the mass of $S$ equals $1-\mu$ and the mass of $J$ equals $\mu$ for some $\mu \in [0,1].$ If the center of mass is located at the origin, then the Hamiltonian of the PCR3BP in a rotating frame is given by
\[
H_{\mu}(q_1, q_2, p_1,p_2) = \frac{1}{2} \lvert p \rvert^2 - \frac{1-\mu}{\lvert q-S\rvert} - \frac{ \mu}{\lvert q - J \rvert}  - p_2 q_1 + p_1 q_2,
\]
where $S = ( -\mu, 0)$ and $J = (1-\mu,0)$ denote the positions of $S$ and $J,$ respectively. Note that in the case $\mu=0,$ this becomes the Hamiltonian of the rotating Kepler problem, see
\eqref{eq:RKPHAM}.

We assume $\mu < \frac{1}{2},$ so that $S$ is heavier than $J.$
 The Hamiltonian $H_{\mu}$ admits five equilibria $L_j, j=1,2,3,4,5$, ordered by action in the following way: 
 \[
 H_{\mu}(L_1) <H_{\mu}(L_2)<H_{\mu}(L_3)<H_{\mu}(L_4)=H_{\mu}(L_5). 
 \]
These values tend to $-\frac{3}{2},$ the critical value of the rotating Kepler problem, as $\mu \to 0.$ We are interested in energies slightly above the first two critical values $H_\mu(L_1)$ and $H_\mu(L_2).$    The corresponding  Hill's regions are illustrated in Figure \ref{HIllregion}.
           \begin{figure}[ht]
  \centering
  \includegraphics[width=0.8\linewidth]{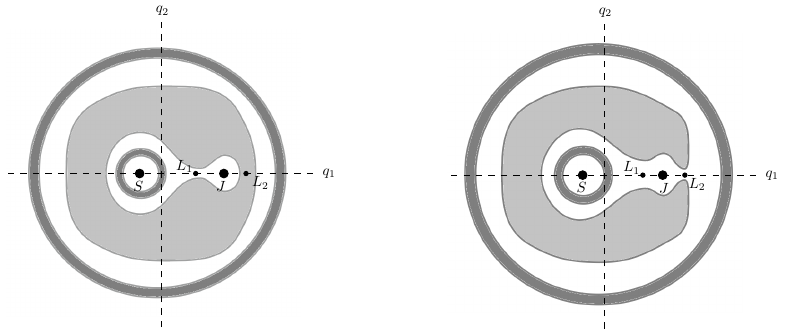}
 \caption{Hill's regions corresponding to energies slightly above $H_\mu(L_1)$ (left) and $H_\mu(L_2)$ (right). The darkly shaded regions indicate the projection of invariant tori.
  }
 \label{HIllregion}
\end{figure}
 Since $L_1$ and $L_2$ are of saddle-center type, a theorem of Lyapunoff \cite{Lia47} shows that 
 for energy slightly above $H_\mu(L_j),$ the energy level carries a unique hyperbolic periodic orbit near $L_j,$ called the Lyapunoff orbit, $j=1,2.$

In his dissertation \cite{McGehee}, McGehee   established the presence of invariant tori on the energy level, as in Figure \ref{HIllregion}, for $\mu >0$ small enough, in other words, for $J$   sufficiently light.     By making use of these tori, he was able to find a  homoclinic orbit to the Lyapunoff orbit associated with $L_1$ or $L_2.$  See Figure \ref{Fig:hetero}.

Using McGehee's result, Koon, Lo, Marsden and Ross provided a numerical demonstration of the existence of a homoclinic-heteroclinic chain in the Sun-Jupiter system \cite{MR1765636}: let the mass $\mu$ of $J$ be that of Jupiter and the    energy   be that of comet Oterma,  so that it is slightly above $H_\mu(L_2).$ 
For these values, the authors showed that the energy level carries the Lyapunoff orbits near $L_1$ and $L_2$ and a heteroclinic orbit between the two Lyapunoff orbits. These three periodic orbits consist of a so-called homoclinic-heteroclinic chain, which might be used to design spacecraft orbits exploring the interior and exterior regions as illustrated in   Figure \ref{Fig:hetero}.
\begin{figure}[h]
  \centering
  \includegraphics[width=0.4\linewidth]{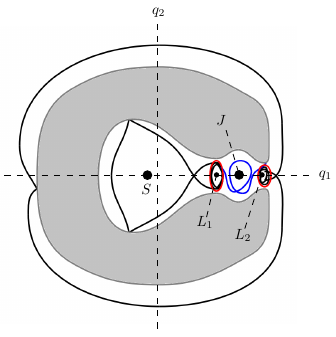}
 \caption{A homoclinic-heteroclinic chain in the PCR3BP. The inner and outer black curves denote  homoclinic orbits to the Lyapunoff orbits (red curves) near $L_1$ and $L_2$, respectively. The blue curve indicates a heteroclinic orbit.}
\label{Fig:hetero}
\end{figure}

Our ultimate goal is to  recover and refine McGehee's results   and   to establish a theoretical foundation to the numerical demonstration by Koon-Lo-Marsden-Ross, using   pseudoholomorphic curves. 
 As compactness is crucial in pseudoholomorphic curve theory,  in order to achieve this goal, one first has to find an efficient way to deal with  non-compact energy levels.   This paper would  provide  a way to construct   finite energy foliations  of     regions in the non-compact energy level in the case $\mu=0,$~i.e.\ in the rotating Kepler problem, so it  would be a first step towards our final goal.

   \medskip
   
   \noindent
   {\bf Outline of the paper.}  We start by recollecting relevant facts about Reeb dynamics and pseudoholomorphic curves in symplectisations in    Section \ref{sec:basic}. Then in    Section \ref{sec:RKP} we review results on the rotating Kepler problem that will be needed in our argument later on. In     Section \ref{sec:coord} we recall the definition and some properties of Poincar\'e's coordinates, a modification of the well-known Delaunay coordinates that are not valid for  the circular orbits.  Constructions of   finite energy folations and transverse foliations are provided in Sections   \ref{sec:lower} and   \ref{sec:higher}.  We present in     Appendix \ref{sec:app}  an alternative  way to obtain a transverse foliation whose pages are annuli.

      \section{Basic notions in contact geometry}\label{sec:basic}

\subsection{Periodic orbits}

 Let $\Sigma=K^{-1}(0) \subset \R^4$ be a regular energy level of a Hamiltonian $K \colon \R^4 \to \R.$  We endow $\R^4$ with coordinates $(x_1, x_2, y_1, y_2).$
  A periodic   orbit (of $K$) on $\Sigma$ will be denoted by a  pair $P=(w,T),$ where $w \colon \R \to \Sigma$ solves the differential equation $\dot w  = X_K \circ w,$ where $X_K$ denotes the Hamiltonian vector field of $K,$  defined as $\omega_0 (X_K , \cdot) = -dK,$   and $T>0$ is a period. Here, $\omega_0 = dy_1 \wedge dx_1 + dy_2 \wedge dx_2$ denotes the standard symplectic form on $\R^4.$ By abuse of notation, we may also denote by $P$ the trace $w(\R) \subset \Sigma  .$
If $T$ is   minimal, then $P$ is called simple.  Throughout the paper, when we consider a periodic orbit, then we tacitly assume that it  is simple.  
If two periodic orbits $P_1=(w_1, T_1)$ and $P_2=(w_2, T_2)$ satisfy $w_1(\R) = w_2(\R)$ and $T_1 = T_2,$ then they are identified, and we write $P_1 =P_2.$

Suppose that $K$ is invariant under an anti-symplectic involution $\rho \colon \R^4 \to \R^4$, meaning that $\rho$ is an involution satisfying $\rho^* \omega_0 = -\omega_0. $  The Hamiltonian vector field satisfies $\rho_* X_K = -X_K,$ so that 
\[
\phi_K^t = \rho \circ \phi_K^{-t} \circ \rho,
\]
where $\phi_K^t$ indicates the flow of $X_K, $ called the Hamiltonian flow of $K$. This shows that if $P=(w,T)$ is a periodic orbit, then $P_\rho = (w_\rho, T),$ defined as
\begin{equation*}\label{eq:symmetrrhods}
    w_\rho \colon \R \to \Sigma, \quad w_\rho (t) =  \rho \circ w(T-t),
\end{equation*}
is a periodic orbit as well. In the case that $w(\R) = w_\rho(\R),$ we write $P = P_\rho$ and call it a symmetric periodic orbit (with respect to $\rho$). Note that every symmetric periodic orbit intersects the fixed point set ${\rm Fix}(\rho),$ which is assumed to be non-empty, precisely at two points.

Assume that $\Sigma$ is compact and of contact type, i.e.\ there is a Liouville vector field $Y$ that is transverse to $\Sigma.$ The one-form $\lambda = \omega_0(Y, \cdot)$ restricts to a contact form   on $\Sigma,$ still denoted by $\lambda.$ Denote by $R=R_\lambda$ the   Reeb vector field of $\lambda.$  Note that the Hamiltonian vector field $X_K$ and the Reeb vector field $R$ are related by
\[
R = \frac{1}{\lambda(X_K)}X_K,
\]
implying that their flows coincide up to reparametrisation. The Reeb period of a periodic orbit $P $  is defined as the period of $P$ with respect to the Reeb flow.

 If the Liouville vector field $Y$ is invariant under $\rho$, i.e.\ it satisfies $\rho_* Y  =Y,$ then $\rho$ becomes exact, meaning that  it satisfies $\rho^* \lambda = - \lambda.$ 
 In this case, $\rho$  restricts to an anti-contact involution on the contact manifold $(\Sigma, \lambda),$ again denoted by $\rho.$ The triple $(\Sigma, \lambda, \rho)$ is said to be a real contact manifold. 
Note that the Reeb vector field $R$ and the Reeb flow $\phi_R^t$ satisfy
\begin{equation*}\label{eq:phirtrho}
\rho_* R = -R, \quad     \phi_R^t = \rho \circ \phi_R^{-t} \circ \rho.
\end{equation*}
Therefore,  given a periodic orbit $P=(w,T),$ the Reeb periods of $P$ and $P_{\rho}$ coincide.

\subsection{Conley-Zehnder index} \label{sec:indexcz}

 As before, let $\Sigma=K^{-1}(0) \subset \R^4$ be a compact   energy level of   $K \colon \R^4 \to \R ,$ equipped with a transverse Liouville vector field $Y.$ The corresponding contact form is denoted again by   $\lambda.$

 For every $z =(x_1, x_2, y_1, y_2) \in   \Sigma  $ the tangent space $T_z \Sigma $  is spanned by the orthogonal vectors $X_1, X_2, X_3$, defined as
      \begin{equation}\label{eq:X123}
      X_j = A_j  \nabla K    , \quad j=1,2,3,
      \end{equation}
      where the $4 \times 4$ matrices $A_j, j=1,2,3,$ are given by
$$
A_1 = \left(\begin{array}{cc}  0 & J \\ J & 0 \end{array} \right), \ \ \  A_2 =\left(\begin{array}{cc} J & 0 \\ 0 & -J \end{array} \right) ,\ \ \ A_3 = \left(\begin{array}{cc} 0 & I \\ -I & 0 \end{array} \right),
$$
with 
$$
  I = \left(\begin{array}{cc} 1 & 0 \\ 0 & 1 \end{array} \right),  \ \ \ \ J = \left(\begin{array}{cc} 0 & 1 \\ -1 & 0 \end{array} \right).
$$
Note that   $X_3=A_3 \nabla K = X_K$ is parallel to the Reeb vector field $R.$   The projection $\pi \colon T\Sigma \to \xi = \ker \lambda$ along $R$ restricted to the tangent plane distribution ${\rm span} \{X_1, X_2\}$ induces an isomorphism from  ${\rm span} \{X_1, X_2\}$ to the contact structure $\xi .$ In particular, the latter is   spanned by the vector fields
      \begin{align}\label{eq:trixi}
      \begin{split}
      \overline X_1 &= \pi(X_1) = X_1 - \lambda(X_1)R,\\
      \overline X_2 &= \pi (X_2) = X_2 - \lambda(X_2)R.
      \end{split}
      \end{align}

Let $\mathfrak{T}$ denote the unitary trivialisation of $T \Sigma / \R X_3,$ induced by $X_1$ and $X_2.$ Set
\[
\kappa_{ij} := \left< {\rm Hess}K \cdot X_i, X_j \right>, \quad i,j=1,2,3.
\]
 In the trivialisation $\mathfrak{T}$ the transverse linearised flow of the Hamiltonian vector field $X_K$ along a  periodic trajectory $P=(w,T)$ on $\Sigma$ is described by solutions $ \alpha(t)=(\alpha_1(t), \alpha_2(t)) \in \R^2$ to the ODE
 \begin{equation}\label{eq:ODEODE}
 \begin{pmatrix} \dot{\alpha}_1(t) \\ \dot{\alpha}_2(t) \end{pmatrix} = \begin{pmatrix} -\kappa_{12} & - \kappa_{22} - \kappa_{33} \\ \kappa_{11} + \kappa_{33} & \kappa_{12} \end{pmatrix} \bigg|_{w(t)} \begin{pmatrix} \alpha_1(t) \\ \alpha_2(t) \end{pmatrix}.  
 \end{equation}

 Let $\theta(t)$ be any  continuous argument of a non-vanishing solution  $\alpha(t) ,$ i.e.\ $\alpha_1(t)+ i\alpha_2(t) \in \R_+ e^{i \theta(t)}.$ Define the rotation interval of $P$ as
 \begin{equation} \label{def:rotin}
 I = \{ \theta(T) - \theta(0)  \mid \alpha(0)\neq 0 \} ,
 \end{equation}
which  depends only on the initial condition $\alpha(0)\neq 0.$ It is a compact   connected interval with   length strictly less than $\pi.$  The periodic orbit $P$ is non-degenerate if and only if $\partial I \cap 2 \pi \Z  = \emptyset.$
Given $\varepsilon>0$ small enough, we set $I_\varepsilon:= I - \varepsilon.$ Then the Conley-Zehnder index of $P$ is defined as
 \[
 \mu_{\rm CZ}(P) = \begin{cases}   2k+1 & \mbox{ if }    I_\varepsilon \subset ( 2k\pi, 2(k+1)\pi), \\ 2k & \mbox{ if } 2k\pi \in {\rm int}(I_\varepsilon).\end{cases}
 \]
 Note that in this definition the Conley-Zehnder index is lower semi-continuous with respect to the $C^0$-topology.

Suppose that the Hamiltonian $K$ is invariant under an   anti-symplectic involution $\rho \colon \R^4 \to \R^4, $ and the Liouville vector field $Y$ satisfies $\rho_*Y=Y,$ so that $\rho$   restricts to an anti-contact involution of $\Sigma,$ denoted again by $\rho$, as in the previous section.    Since
\begin{equation*}\label{eq:rhoinvarintx1x2}
    \rho_* X_j = X_j, \quad j =1,2,3,
\end{equation*}
 we find by the definition of the Conley-Zehnder index that
\begin{equation*}\label{eq:indexofrhoorbit}
    \mu_{\rm CZ}(P_\rho) = \mu_{\rm CZ}(P).
\end{equation*}




\subsection{Pseudoholomorphic curves in symplectisations}\label{sec:pseudoho}

Let $(\Sigma,\lambda)$ be a closed contact three-manifold.    We denote by $\mathcal{J}(\lambda)$ the  set of $d\lambda$-compatible almost complex structures on $\xi=\ker \lambda$. We extend each $J \in \mathcal{J}(\lambda)$ to a  $d(e^r\lambda)$-compatible almost complex structure $\tilde J$ on $T(\R \times \Sigma)$ that is given by
 \begin{equation}\label{eq:SFTJ}
\tilde J \colon \partial_r \mapsto R, \quad \tilde J|_{\xi}=J  ,
\end{equation}
where $r$ denotes the coordinate on $\R.$ Note that $\tilde J$ is $\R$-invariant.  

Given a closed Riemann surface $(S, j)$ and a finite set $\Gamma \subset S,$ we consider a smooth map $\tilde u = (a,u) \colon  S \setminus \Gamma \to   \R \times \Sigma,  $ satisfying   $d\tilde u \circ j = \tilde J \circ d \tilde u$ and   a finite energy condition
\[
0< E(\tilde u):= \sup_{\phi} \int_{S \setminus \Gamma}  \tilde u^* d(\phi \lambda)<\infty,
\]
where the supremum is taken over all monotone increasing smooth functions $\phi \colon  \R \to [0,1].$   In this paper we only consider the case $S=S^2$ and $\# \Gamma \in \{ 1, 2 \}. $ If $ \# \Gamma = 1$ or $\# \Gamma=2,$ then such a map is called a finite energy plane or a finite energy cylinder, respectively.

Points in $\Gamma$ are called punctures of $\tilde u= (a,u).$ 
If $a$ is bounded in a small neighbourhood of a puncture $z_0 \in \Gamma,$ then $z_0$ is called removable. In this case, $\tilde u$ can be smoothly extended over $z_0$. See \cite{Hofer93}. Otherwise, either $a(z) \to +\infty$ or $a(z) \to -\infty$ as $z \to z_0. $ Then the puncture $z_0$ is called positive or negative, respectively. In the following we assume that every   puncture is not removable, so that the set $\Gamma$ is decomposed into $\Gamma = \Gamma^+ \sqcup \Gamma^-, $ where $\Gamma^\pm$ consist only of positive/negative punctures, respectively. We assign $\varepsilon \in \{ \pm 1\}$ to each puncture $z_0 \in \Gamma^\pm$ according to its sign.

The following statement is taken from  \cite{Hofer93} and \cite{HWZI}. 
\begin{theorem} 
Given $z_0 \in \Gamma$   choose a holomorphic chart $\varphi \colon ( \mathbb{D} \setminus \partial \mathbb{D}, 0) \to (\varphi(\mathbb{D} \setminus \partial \mathbb{D}), z_0)$ centred at $z_0.$ Abbreviate $\tilde u(s,t) = \tilde u \circ \varphi( e^{-2\pi(s+it)}), (s,t) \in [0,+\infty) \times \R / \Z.$ Then every sequence $s_n \to+\infty $ admits    a subsequence $s_{n_k}$ and a $\tau$-periodic  orbit $x$ of the Reeb vector field   such that $u(s_{n_k}, t) \to x(\varepsilon \tau t+d)   $ in $C^{\infty} (\R / \Z, \Sigma)$ as $k \to +\infty,$ for some $d \in \R.$  In the case  $x$ is non-degenerate,   $u(s , t) \to x(\varepsilon \tau t+d)   $ in $C^{\infty} (\R / \Z, \Sigma)$ as $s \to +\infty.$

 \end{theorem}

  The periodic orbit $x$ in the theorem above is called an asymptotic limit of $\tilde u$ at the puncture  $z_0 \in \Gamma.$  If the puncture $z_0$ is positive or negative, the corresponding asymptotic limit is also called positive or negative, respectively. 
A finite energy curve is called non-degenerate if its every asymptotic limit is non-degenerate. The theorem above shows that a non-degenerate finite energy plane admits a unique asymptotic limit.

We now assume that $(\Sigma, \lambda)$ is equipped with an anti-contact involution $\rho.$ It defines the exact anti-symplectic involution $\tilde \rho := {\rm Id}_\R \times \rho$ on $(\R \times \Sigma, d( e^r \lambda)),$   meaning that $\tilde \rho^* ( e^r \lambda) = -e^r \lambda.$  An almost complex structure $J \in \mathcal{J}(\lambda)$ is said to be $\rho$-anti-invariant if
\begin{equation}\label{eq:Jrhoanti}
    D\rho|_\xi \circ J \circ ( D\rho|_\xi)^{-1}=-J.
\end{equation}
Denote by $\mathcal{J}_\rho(\lambda) \subset \mathcal{J}(\lambda)$ the set of $d\lambda$-compatible and $\rho$-anti-invariant almost complex structures on $\xi.$ If $J \in \mathcal{J}_\rho(\lambda),$   then the associated $d(e^r \lambda)$-compatible almost complex structure $\tilde J$ as in \eqref{eq:SFTJ} is $\tilde \rho$-anti-invariant. 
It follows that if $\tilde u = (a,u)$ is a finite energy $\tilde J$-holomorphic   curve, then $\tilde u_{\rho} := \rho \circ \tilde u \circ I,$ where $I(z) = \bar z$, is a finite energy $\tilde J$-holomorphic curve with Hofer energy $E(\tilde u_\rho) = E(\tilde u).$ Moreover, if   $x$ is a positive or negative asymptotic limit of $\tilde u$ at $z_0 \in \Gamma$, then   $x_\rho $ is a positive or negative asymptotic limit of $\tilde u_\rho $ at $\bar{z}_0 \in I(\Gamma),$ respectively.  If $\tilde u = \tilde u_\rho,$ then  it is said to be invariant (with respect to $\rho$). Note that every asymptotic limit of  invariant finite energy planes and cylinders has to be a symmetric periodic orbit. 
 For more information on the behaviour of pseudoholomorphic curves under the symmetry, we refer the reader to \cite{FK16, Kim20}.





\subsection{Transverse  foliations}

Let $\psi^t$  be a smooth flow on an oriented closed three-manifold $\Sigma.$ A transverse foliation $\mathcal{F}$ of $\Sigma$ adapted to $  \psi^t $ is formed by the following data:
\begin{enumerate}
    \item the singular set $\mathcal{P}$ which   consists of finitely  many simple periodic  orbits,  called binding orbits.

    \item a smooth foliation of $\Sigma \setminus \cup_{P \in  \mathcal{P}}P$ by properly embedded   surfaces  transverse to the  flow. Each leaf $\mathring{F} \in \mathcal{F}$ has an orientation induced by the flow and the orientation of $\Sigma$. The closure $F$ of $\mathring{F}$ is a compact embedded surface in $\Sigma$ whose boundary is a subset of $\mathcal{P}.$ Each end $z$ of $\mathring{F}$ is called a puncture. Every puncture $z$ is associated with a binding orbit $P_z \in \mathcal{P},$ called the asymptotic limit of $\mathring{F}$ at $z.$  The asymptotic limit $P_z$ at $z$ has two orientations, one induced by $\mathring{F}$ and the other induced by the flow. If the two orientations coincide with each other, then the puncture $z$ is said to be positive. Otherwise it is negative. 
    \end{enumerate}

In the case where $\Sigma$ has a non-empty boundary that is tangent to the flow, then we extend the definition of a transverse foliation in such a way that     every binding orbit is required to be contained in the interior of $\Sigma$ and each regular leaf is transverse to the boundary of $\Sigma$.


Assume that the flow $\psi^t$ is anti-invariant under a smooth involution $\rho \colon \Sigma \to \Sigma,$ meaning that $\psi^t = \rho\circ \psi^{-t} \circ \rho.$ A transverse foliation $\mathcal{F}$     is said to be symmetric (with respect to $\rho$) if  
$P_\rho \in \mathcal{P},$ $\forall P \in \mathcal{P},$ and $\mathring{F}_\rho :=\rho(\mathring{F}) \in \mathcal{F},$ $\forall \mathring{F} \in \mathcal{F}.$

\subsection{Finite energy foliations}  
Let $\lambda$ be a   contact form on the tight three-sphere $(S^3, \xi_0) $ and   $J \in \mathcal{J}(\lambda).$ 

\begin{definition}
A (stable) finite energy foliation for $(S^3, \lambda, J)$ is a two-dimensional smooth foliation $\tilde{\mathcal{F}}$ of $\R \times S^3$ satisfying the following properties:
\begin{enumerate}
    \item Every leaf $\tilde F \in \tilde{\mathcal{F}}$ is the image of an embedded finite energy $\tilde J$-holomorphic curve $\tilde u_{\tilde F}=(a_{\tilde F}, u_{\tilde F})  $ having  only one positive puncture. The number of negative punctures  is any finite number.  The energies of such finite energy surfaces are uniformly bounded.
    
    \item Every asymptotic limit of $\tilde F$  is unknotted and has Conley-Zehnder index $1,2$ or $3,$ 
     where  asymptotic limits of a leaf $\tilde F $ is defined to be    asymptotic limits of $\tilde u_{\tilde F}$.  
    
    \item   A  map $T \colon \R \times \tilde{\mathcal{F}} \to \tilde{\mathcal{F}},$
    \[
    T(c,\tilde F) = \{ (c + a , z) \in \R \times S^3 \mid (a, z) \in \tilde F \},
    \]
    defines an $\R$-action on $\tilde{\mathcal{F}}.$ If $\tilde F$ is a fixed point of $T$, i.e.\ $T(c, \tilde F) = \tilde F, \forall c \in \R,$  then its Fredholm index ${\rm Ind}(\tilde F):= {\rm Ind}(\tilde u_{\tilde F})$  is equal to $0$ and $\tilde u_{\tilde F}$ is a cylinder over a periodic orbit.  Here, the Fredholm index ${\rm Ind}(\tilde u_{\tilde F})$ represents the local dimension of the moduli space of unparametrised $\tilde J$-holomorphic curves near $\tilde u_{\tilde F}$ having the same asymptotic limits.         In the case $\tilde F$ is not a fixed point, then  ${\rm Ind}(\tilde F) \in \{ 1, 2\} $ and  $u_{\tilde F}$ is embedded and transverse to the Reeb flow of $\lambda.$ 
\end{enumerate}
\end{definition}

Every non-degenerate contact form $\lambda$ admits a finite energy foliation, as stated below.

\begin{theorem}[{Hofer, Wysocki and Zehnder \cite{HWZ03foliation}}]
 Let $\lambda$ be a non-degenerate contact form on $(S^3, \xi_0). $ Then for a generic $J,$ there is a   finite energy foliation for $(S^3, \lambda, J).$
\end{theorem}

The projection of a finite energy foliation, obtained in the theorem above, to $S^3$ provides a transverse foliation  adapted to the Reeb flow of $\lambda.$ 

Suppose that $(S^3, \lambda)$ admits an anti-contact involution $\rho$ and that $J \in \mathcal{J}_\rho(\lambda).$ A finite energy foliation $\tilde{\mathcal{F}}$ is said to be symmetric (with respect to $\rho$) if $\tilde{\rho}(\tilde{\mathcal{F}}) = \tilde{\mathcal{F}}.$ The projection of a symmetric finite energy foliation to $S^3$ produces a symmetric transverse foliation.

   \section{The rotating Kepler problem}\label{sec:RKP}
Recall that the Hamiltonian of the rotating Kepler problem is given by
   \[ 
   H(q_1, q_2, p_1,p_2) = \frac{1}{2} ( (p_1 +q_2)^2 + (p_2 - q_1)^2) - \frac{1}{\lvert q\rvert}- \frac{1}{2} \lvert q \rvert^2.
\]
  Consider a smooth function
  \begin{equation*}\label{f}
  f(r) = -\frac{1}{ r} - \frac{1}{2}r^2, \quad r>0.
  \end{equation*}
  It  is strictly increasing for $r<1$  and decreasing for $r>1$  and   attains a unique global maximum $-\frac{3}{2}$ at $r=1.$ 
  Therefore,   for every $c<-\frac{3}{2}$ the equation $f(r)=c$ has precisely two roots $ r_c^b < 1 < r_c^u.$  It follows that for every $c<-\frac{3}{2}$  the energy level $H^{-1}(c)$ consists of two components, $\Sigma_c^b$ that projects to   $\{ q \in \R^2 \setminus \{ 0 \} : \lvert q \rvert \leq r_c^b\}$ and   $\Sigma_c^u$  that projects to    $\{ q \in \R^2 \setminus \{ 0 \} : \lvert q \rvert \geq r_c^u\}.$ 
           If $ -\frac{3}{2} < c < 0$, then  $H^{-1}(c)$ is connected and projects to $\R ^2 \setminus \{ 0 \}.$

   Since $\{ E, L \} =0$, 
where the Kepler energy $E$ and the angular momentum $L$ are as in \eqref{eq:ELoriginal},   
   the Hamiltonian flows of $E$ and $L$ commute:
   \begin{equation*}\label{eq:flowcommute}
   \phi_H^t = \phi_{E-L}^t= \phi_E^t \circ \phi_{-L}^t = \phi_{-L}^t \circ \phi_E^t = \phi_L^{-t} \circ \phi_E^t.
   \end{equation*}
See, for instance, \cite[Theorem 3.1.7]{FvK18book}.
      Therefore,  given a   trajectory $\gamma(t)$ of $H$, i.e.\ a trajectory of the Hamiltonian vector field $X_H,$ there exists a Kepler ellipse $\alpha(t) $ such that
   \begin{equation}\label{eq:gammaalpha}
   \gamma(t) = \exp(- it)\alpha(t).
   \end{equation}
   Note that $\gamma$ is not necessarily periodic.  It is periodic if and only if the period of $\alpha$ and $2 \pi $    are $\Q$-commensurable. In other words, if we denote the period of $\alpha$ by $T>0$, then $\gamma$ is periodic if and only if 
   \begin{equation}\label{eq:torusorbitperiod}
   kT = 2 \pi l \quad \text{ for some coprime positive integers } k,l.
   \end{equation}

   Let $\alpha(t)$ be a Kepler ellipse   that projects to a circular orbit. Then  $\gamma(t),$ as in \eqref{eq:gammaalpha}, is circular as well. This happens if and only if $X_E$ and $X_L$ are parallel along $\alpha(t)$.    In order to study circular orbits, recall  that the eccentricity $e$ of a Kepler ellipse satisfies
   \begin{equation}\label{eq:e}
   e^2 = 2EL^2 +1.
   \end{equation}
  Since circular orbits have $e=0$, we obtain
  \begin{equation}\label{eq:ecc}
  0 = 2 E ( H-E)^2+1.
  \end{equation}
  
Fix $H=c$. If $c< -\frac{3}{2}$, then equation \eqref{eq:ecc} has three solutions 
\[
E_{\rm retro}^b < E_{\rm direct}^b < - \frac{1}{2} < E_{\rm direct}^u <0,
\]
implying that $H^{-1}(c)$ carries precisely three circular orbits:    $\gamma_{\rm retro}^b$  with  $E = E_{\rm retro}^b$ and  $\gamma_{\rm direct}^b$ with $E = E_{\rm direct}^b,$ both lying  on   $\Sigma_c^b,$  and  $\gamma_{\rm direct}^u$ with $E = E_{\rm direct}^u,$ lying  on   $\Sigma_c^u.$ Their angular momenta are
\begin{equation*}\label{eq:Lnumbers}
L _{\rm retro}^b  =  - \frac{1}{\sqrt{ -2 E_{\rm retro}^b}}, \quad L _{\rm direct}^b  =   \frac{1}{\sqrt{ -2 E_{\rm direct}^b}}, \quad L _{\rm direct}^u  = \frac{1}{\sqrt{ -2 E_{\rm direct}^u}},
\end{equation*}
respectively.
    Note that our coordinate system rotates in the counterclockwise direction. Therefore, $\gamma_{\rm retro}^b$ rotates in opposite direction to the coordinate system and $\gamma_{\rm direct}^b$ and $\gamma_{\rm direct}^u$ rotate in the same direction with the coordinate system. For this reason, $\gamma_{\rm retro}^b$ is called the retrograde circular orbit and $\gamma_{\rm direct}^b$ and $\gamma_{\rm direct}^u$ are called the direct circular orbits.

   Assume that $ -\frac{3}{2} < c < 0$.   In this case, equation  \eqref{eq:ecc} has a unique solution $ E = E_{\rm retro} ^u\in (-2,0)$, implying that $H^{-1}(c)$  carries a unique circular orbit, denoted by $\gamma_{\rm retro}^u$,
  whose  angular momentum is  given by
 \begin{equation*}\label{eq:Lnumberss}
 L _{\rm retro}^u  =  -  \frac{1}{\sqrt{ -2 E_{\rm retro}^u}} \in ( -\frac{1}{\sqrt[3]{2}}, - \frac{1}{2}).
 \end{equation*}
    This unique circular orbit is retrograde.  See Figure \ref{fig1}.

          \begin{figure}[ht]
  \centering
  \includegraphics[width=1.0\linewidth]{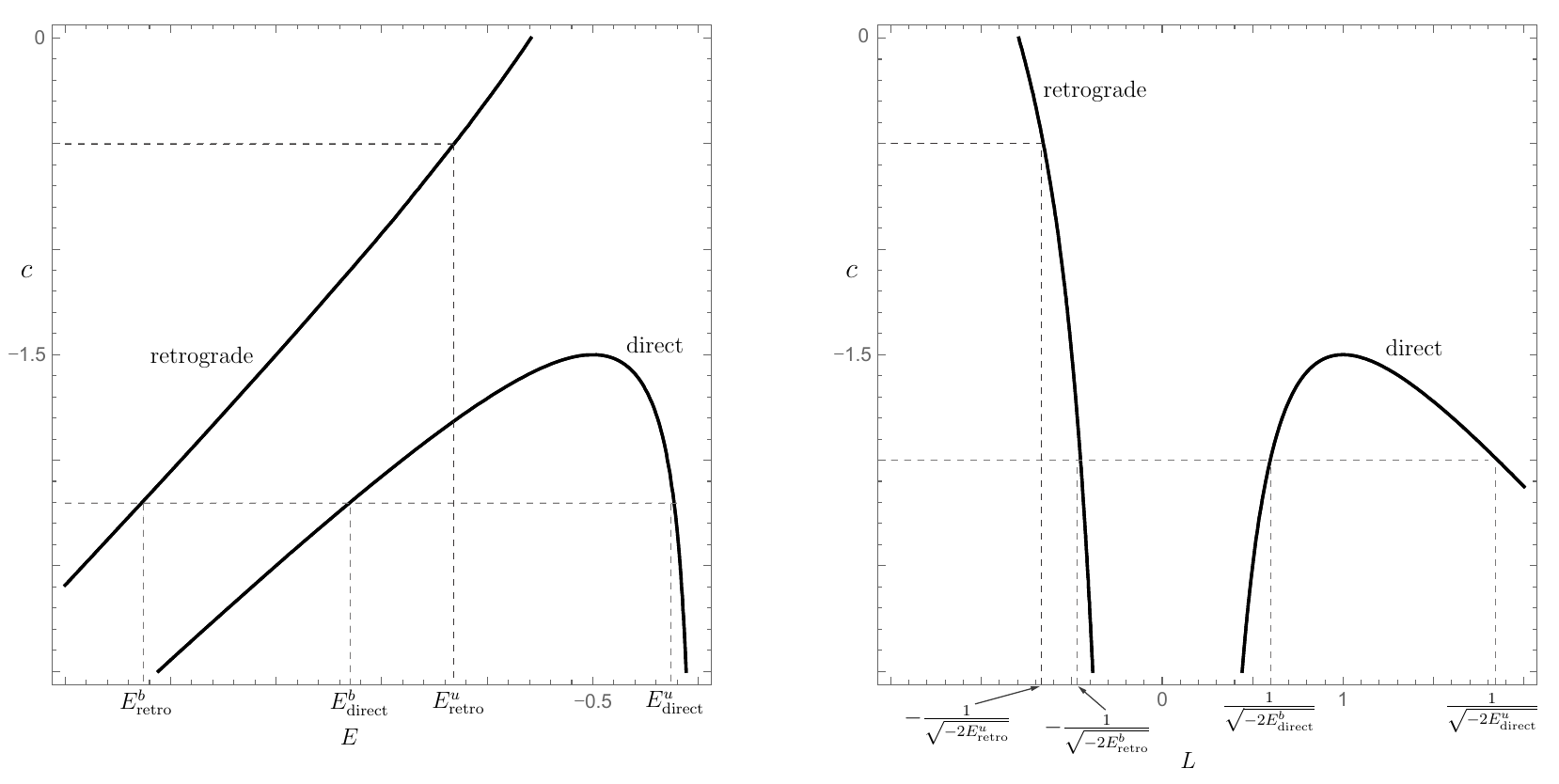}
 \caption{  The solution set of (left) $2E(c-E)^2+1=0$; (right)  $2(c+L)L^2+1=0$ is the union of two graphs. The left and right curves indicate the retrograde and direct circular orbits, respectively. }
\label{fig1}
\end{figure}

   Suppose that $\alpha(t)$ is a non-circular Kepler ellipse  and   let $\gamma(t)$ be the corresponding periodic  trajectory of $H$,  satisfying \eqref{eq:gammaalpha} and  \eqref{eq:torusorbitperiod}.
      Since $L$ generates rotation, it implies that $\gamma(t)$ lies in an $S^1$-family of  periodic   trajectories of $H$ satisfying  the same  condition.    For this reason   we call a periodic trajectory of $H$ satisfying \eqref{eq:torusorbitperiod} a $T_{k,l}$-type orbit  and an  $S^1$-family of $T_{k,l}$-type orbits a $T_{k,l}$-torus. 
      
 The Kepler energy    is constant along each $T_{k,l}$-torus. Indeed, Kepler's third law implies
 \[
 T^2 = \frac{ (2\pi)^2}{(-2E)^3},
 \]
 and hence using \eqref{eq:torusorbitperiod} we obtain
 \begin{equation*}\label{Eenergy}
 E_{k,l} := -\frac{1}{2} \left( \frac{k}{l} \right)^{\frac{2}{3}}.
 \end{equation*}
 One can easily check that  in the case $  c<-\frac{3}{2} $ all $T_{k,l}$-type orbits  satisfy   $k>l$ on $\Sigma_c^b$ and $k<l$   on $\Sigma_c^u.$

For a fixed coprime $k,l \in \N,$  identity \eqref{eq:e} determines a one-parameter family of $T_{k,l}$-tori. Either   the eccentricity $e$, the angular momentum $L$ or the total energy $H=c$ can be regarded as the parameter of such a family.   We call this family the $T_{k,l}$-torus family.   See Figure \ref{fig2}.

          \begin{figure}[ht]
  \centering
  \includegraphics[width=0.6\linewidth]{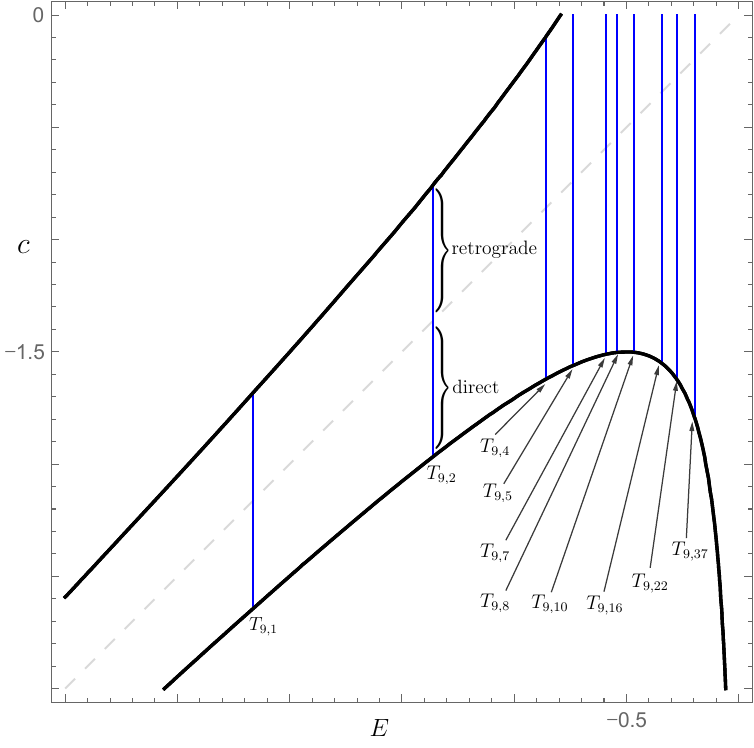}
  \caption{Some $T_{k,l}$-tori with $k=9.$ The dashed line indicates the tori consisting only of collision orbits.}
  \label{fig2}
\end{figure}

 Observe that along every $T_{k,l}$-torus family, there is a unique $T_{k,l}$-torus containing a collision orbit (and hence it consists only of collision orbits). Indeed, a $T_{k,l}$-type orbit $\gamma(t)$ is a collision orbit if and only if a Kepler ellipse $\alpha(t)$ is a collision orbit. This happens precisely at $e=1,$ and in view of \eqref{eq:e} this implies $L=0,$ or equivalently $H=E.$  A $T_{k,l}$-type orbit is said to be direct or retrograde if it    satisfies $L>0$ or $L<0,$ respectively.

 We summarize the discussion so far in the following. Recall that we consider orbits with $E<0$.
 
    \begin{lemma}\label{lem:EL}
    Let $c < -\frac{3}{2}$.      When projected to $\Sigma_c^b,$ the Kepler energy $E$ satisfies
    \[
    E_{\rm retro}^b \leq E \leq E_{\rm direct}^b <-\frac{1}{2}
    \]
    and the angular momentum $L$ satisfies
    \[
   L_{\rm retro}^b  \leq L \leq L_{\rm direct}^b <1.
    \]
           When projected to $\Sigma_c^u$,  we have   
        \[
-\frac{1}{2}<         E_{\rm direct}^u \leq E <0 , \quad  \ \   1<  L _{\rm direct}^u  \leq L  <-c.
\]
                In the case $-\frac{3}{2} < c < 0$,  we have
                \[
                 -2 <   E_{\rm retro}^u \leq E <0, \quad \ \      -\frac{1}{\sqrt[3]{2}} <  L _{\rm retro}^u  \leq L   <-c.
     \]
     \end{lemma}

      \section{Poincar\'e's coordinates}\label{sec:coord}

            



      


                  We briefly recall the definition   of Poincar\'e's coordinates.  For more details we refer the reader to \cite[Chapter 9]{AM78book} and \cite[Section 8.9]{MR3642697}.

Consider the following open subset of $ T^* ( \R^2 \setminus \{ 0 \})$
\[
            \mathcal{E} = \{ (q,p) \in T^* ( \R^2 \setminus \{ 0 \} ) : E(q,p) < 0  < e(q,p)<1, \ L(q,p)>0 \},
\]     
called the direct elliptical domain (recall that we have $L>0$ along direct orbits, see Section \ref{sec:RKP}),
 and define smooth maps $\alpha, \beta \colon \mathcal{E} \to \R / 2\pi \Z$ and $a \colon \mathcal{E} \to \R$ as follows. Choose $(q,p) \in \mathcal{E}$ and let $\gamma$ denote the Kepler ellipse on which  the point $(q,p)$ lies. Then  $\alpha = \alpha (q,p) $ is defined as the argument of the position $q,$ and $\beta = \beta(q,p) $ is defined to be   the argument of the perihelion, i.e.\ the   point, lying on the projection of $\gamma$ to the $q$-plane, that is nearest to the origin. Finally, $a = a(q,p)$ is defined to be the semi-major axis of the Kepler ellipse $\gamma.$ Note that $a = -\frac{1}{2E},$ where $E<0$ indicates the Kepler energy of $\gamma.$ See  Figure \ref{fig:elements}.  
 
    \begin{figure}[ht]
     \centering
\begin{tikzpicture} 
 \draw[-{Latex[length=2mm]}] (-2,-0.1) to (-2.1,-0.2);

 \begin{scope}[rotate=-150]
 \draw[thick] (1,0) ellipse (2cm and 1cm);
  \filldraw[draw=black, fill=black] (1,-1) circle (0.05cm);
  \draw[thick, blue]  (0,0) to (1,-1);
   \draw[thick, red] (0,0)  to (-1,0);
     \filldraw[draw=black, fill=black] (-1,0) circle (0.05cm);
     \draw [->]( 2, -2 ) [out=120 ,in=300]  to (1,-1.05);
     \node at (2,-2) [above] {\footnotesize the present position};
     \draw[->]  (-2, -1) [out=10, in=170]to (-1.05,-0.05);
     \node at (-2, -1)[right] {\footnotesize perihelion};
     \draw[]  (1,0) to (3,0);
       \filldraw[draw=black, fill=black] (1,0) circle (0.05cm);
       \draw (2.9 , 0.1) [out=140, in=-10 ] to (2.15, 0.3);
              \draw (1.85 , 0.3) [out=190, in=40 ] to (1.1, 0.1);
              \draw[->]   (0,1.5 ) [ out=300, in=140] to    (0.95,0);
              \node at (0,1.5) [right] {\footnotesize center};
\node at (2, 0.3) {\tiny $a$};
  \end{scope}
  \draw[dashed, ->] (-3,0) to (2,0);
  \draw[dashed, ->] (0,-2) to (0,2);
 \node at (0,2) [above] {\small $q_2$};
 \node at (2,0) [right] {\small$q_1$};
   \draw[thick, red] (0.4,0) [out=80, in=330] to (0.3, 0.18);
\draw[thick, blue] (0.2, 0) [out=110, in=50] to (-0.22, 0.075);
 \node[blue] at (-0.15, 0.25)   {\tiny $\alpha$};
 \node[red] at (0.35, 0.15) [right] {\tiny $\beta$};
 
 \begin{scope}[rotate=-150]
  \filldraw[draw=black, fill=black] (1,-1) circle (0.05cm);
     \filldraw[draw=black, fill=black] (-1,0) circle (0.05cm);
       \filldraw[draw=black, fill=black] (1,0) circle (0.05cm);
  \end{scope}
      \end{tikzpicture}
    \caption{Some quantities associated with a Kepler ellipse of Kepler energy $E<0$.   The quantity $a$ is the semi-major axis that is equal to $-\frac{1}{2E}.$ The quantity $\alpha$ is given by the  argument of the present position, and $\beta$ is given by  the argument of the perihelion.  }
    \label{fig:elements}
 \end{figure}
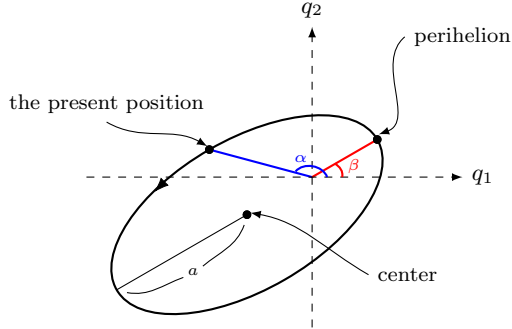  
Set  $ \mathcal{K}=  \R / 2\pi \Z \times  \R / 2\pi \Z \times S,$ where $ S = \{ (x, y) \in \R^2 : x>0, y \in (0,1) \}.$      Then there  is a   diffeomorphism
\[
\mathscr{K} \colon \mathcal{E} \to \mathcal{K} , \quad (q_1,q_2,p_1,p_2) \mapsto (\alpha, \beta, a ,e),
\]
called the Kepler map. This implies that the domain $\mathcal{E}$ is a thickened two-torus.
 
 \begin{theorem}[{Lagrange, 1808}]\label{thm:Lagrange} The Kepler map $\mathscr{K}$ satisfies
\[
\mathscr{K} _* ( dp_1 \wedge dq_1 + dp_2 \wedge dq_2) = d   \sqrt{ a (1-e^2)} \wedge  d \beta +   d \sqrt{a}  \wedge d(\alpha - \beta) .
\]

 \end{theorem}

 Suggested by this result,  Delaunay (1860) considered a diffeomorphism
\[
\mathscr{D} \colon \mathcal{K} \to \mathcal{D}, \quad (\alpha, \beta, a, e) \mapsto ( \beta, \alpha- \beta, \sqrt{ a ( 1-e^2)}, \sqrt{ a}),
\]
 where 
 \[
 \mathcal{D} := \R / 2 \pi \Z \times \R / 2 \pi \Z \times \mathcal{O} , \quad \mathcal{O}=\{ (x,y) \in \R^2 : 0<x<y \}.
 \]
 Write 
 \[
 (l, k, L, K) =  ( \beta, \alpha- \beta, \sqrt{ a ( 1-e^2)}, \sqrt{ a})
 \]
 and define 
 \[
 \Delta = \mathscr{D} \circ \mathscr{K} \colon \mathcal{E} \to \mathcal{D}, \quad (q_1, q_2,p_1,p_2) \mapsto  (l, k, L, K) ,
 \]
 called the Delaunay map.    
 Theorem \ref{thm:Lagrange} implies that $\Delta$ is a symplectomorphism, namely, it satisfies
 \[
 \Delta_* ( dp_1 \wedge dq_1 + dp_2 \wedge dq_2 ) =  d L \wedge dl + dK \wedge dk.
 \]
However, since the argument of the perihelion is undefined for a circular orbit, the   Delaunay map cannot be used to study the dynamics in a neighbourhood of a circular orbit.

 \begin{remark}
 The classical notations of the Delaunay variables are $(g,l,G,L)$, not $(l,k,L,K).$ However, we stick to our choice since we would like to denote the angular momentum by $L$. 
 \end{remark}

 To remedy this problem, Poincar\'e introduced new coordinates as follows. 
Define
\[
\Pi \colon \mathcal{D} \to \R \times \R / 2 \pi \Z \times \R \times \R, \quad (l, k , L, K) \mapsto (\eta, \lambda, \xi, \Lambda),
\]
where
\begin{align*}
\eta &= - \sqrt{ 2(K-L)} \sin l, \\
\lambda &= l + k ,\\
\xi &= \sqrt{ 2 (K-L) } \cos l,\\
\Lambda &= K.
\end{align*}
It turned out that $\Pi$ is a symplectomorphism onto its image $\mathcal{P} = \Pi(\mathcal{D}).$  
We then define  the Poincar\'e mapping
\[
\mathscr{P} = \Pi \circ \Delta = \Pi \circ \mathscr{D} \circ \mathscr{K} \colon \mathcal{E} \to \mathcal{P},
\]
which is a   symplectomorphism as well. Namely, it satisfies 
      \[
      \mathscr{P} ^* ( d\xi \wedge d\eta + d\Lambda \wedge d\lambda) =  dp_1 \wedge dq_1 + dp_2 \wedge dq_2  .
      \]

      We shall now add the circular orbits.  Recall that for the circular orbits we have $e=0, \alpha = \beta$ and $L=K.$
      Let
       \begin{align*}
       \overline{\mathcal{E}} &= \mathcal{E} \cup  \{ (q,p) \in T^* ( \R \setminus \{ 0 \} ) : E(q,p) < 0  =  e(q,p),  \ L(q,p)>0 \} \\
       &=      \{ (q,p) \in T^* ( \R \setminus \{ 0 \} ) : E(q,p) < 0  \leq  e(q,p)<1, \ L(q,p)>0 \}
       \end{align*}
and 
\[
       \overline{\mathcal{P}} = \mathcal{P} \cup  \{   ( 0, \lambda, 0, \Lambda) : \lambda \in \R / 2 \pi \Z, \Lambda>0 \} .
       \]
              If           $e(q_1, q_2, p_1,p_2) =0,$ then we set
              \[
           {\mathscr{P}} (q_1, q_2, p_1,p_2) = (0, \lambda, 0, \Lambda)
            \]
         by defining $\lambda =\alpha$ and $\Lambda  = L=K.$ 
         This extends the Poincar\'e mapping to a symplectomorphism   $ {\mathscr{P}} \colon        \overline{\mathcal{E}}  \to      \overline{\mathcal{P}}.$ 
          We conclude that in the study of direct  orbits, i.e.\ orbits  with $L>0,$ we can work with the transformed Hamiltonian $ {\mathscr{P}}_*H \colon \overline{\mathcal{P}} \to \R.$

We now consider  
\[
           \overline{\mathcal{E}}' = \{ (q,p) \in T^* ( \R^2 \setminus \{ 0 \} ) : E(q,p) < 0  \leq  e(q,p)<1, \ L(q,p)<0 \} 
\]     
and note that $\psi ( \overline{\mathcal{E}}') =\overline{\mathcal{E}},$ where $\psi(  q_1,q_2,p_1,p_2) = (q_ 1,q_2,-p_1,-p_2).$  Define $\phi \colon \overline{\mathcal{P}} \to  \overline{\mathcal{P}}':= \phi(\overline{\mathcal{P}})$ by
\[
\phi( \eta, \lambda ,\xi, \Lambda) = (\eta, \lambda, -\xi, -\Lambda).  
\]
Then we obtain a symplectomorphism 
\[
\mathscr{P}':=\phi \circ \mathscr{P} \circ \psi \colon \overline{\mathcal{E}}' \to \overline{\mathcal{P}}' .
\]
  It follows that        in the study of  retrograde orbits, i.e.\ orbits with $L<0,$ we can work with the transformed Hamiltonian $ \mathscr{P}'_*H \colon \overline{\mathcal{P}}' \to \R.$

      \section{Below the critical energy}\label{sec:lower}

      Throughout the section we fix   $c<-\frac{3}{2} $ and  $ E_0 \in (E_{\rm direct}^u, 0). $ We shall prove the assertion of  Theorem \ref{thm:main} for the set
\begin{equation*}\label{defset1}
      \Sigma _{\rm direct}^u = \{\text{all periodic orbits of $X_H$ with   $E \in [E_{\rm direct}^u, E_0]$}\} \subset \Sigma_c^u,
     \end{equation*}    
   diffeomorphic to a solid torus with   core being the direct circular orbit $\gamma_{\rm direct}^u.$  Note that it   contains only direct orbits and  no collision orbits. Indeed, every periodic orbit on $\Sigma_{\rm direct}^u$ has $L>0.$ See Lemma \ref{lem:EL}.

               \begin{remark}\label{rmk:large}
The neighbourhood $\Sigma^u_{\rm direct}$ of $\gamma^u_{\rm direct}$ could be arbitrarily large: just take $E_0 \in (E^u_{\rm direct}, 0)$ to be sufficiently close to $0.$
               \end{remark}
       
   We shall work in Poincar\'e's coordinates. 
              For simplicity we write $(x_1, x_2, y_1, y_2) = (\eta, \lambda, \xi, \Lambda).$ 
      Note that 
\[     
 y_2 = \sqrt{a} = \frac{1}{\sqrt{-2E}}, \quad       x_1^2 +y_1^2  = 2(y_2-L). 
      \]
      Then the   Hamiltonian \eqref{eq:RKPHAM}   is given in Poincar\'e's coordinates by
      \begin{equation}\label{eq:orihamlow}
      H (x_1, x_2 , y_1, y_2) = - \frac{1}{2y_2^2} - y_2 + \frac{1}{2} ( x_1 ^2 + y_1 ^2)  
      \end{equation}
      for $(x_1, x_2, y_1, y_2) \in   \overline{\mathcal{P}} \subset \R \times \R/ 2 \pi \Z \times \R \times \R_+.$ 
          We set
          \begin{equation}\label{eq:H1H2}
  H_1(x_1, y_1) = \frac{1}{2}(x_1^2 + y_1^2), \quad H_2(  y_2) = - \frac{1}{2y_2^2}-y_2,
  \end{equation}
           so that $H = H_1 + H_2.$

           Since   $E_0 \in ( E_{\rm direct}^u, 0)$ is fixed and we consider the subset $\Sigma_{\rm direct}^u  \subset \Sigma_c^u,$ we have 
           \[
           y_2 \in [\Lambda_1, \Lambda_2]  \quad \; \; \mbox{ on} \; \;   \Sigma _{\rm direct}^u  ,
           \] where  $\Lambda_1 =  \sqrt{ -\frac{1}{2E_{\rm direct}^u}}$ and    $\Lambda_2 = \sqrt{ - \frac{1}{2E_0}}.$ 
Since $E_{\rm direct}^u >- \frac{1}{2},$  see Lemma \ref{lem:EL}, we have $\Lambda_1>1.$

\subsection{Disc foliation}    \label{sec:discfol}

  It is easy to construct a disc foliation of $\Sigma_{\rm direct}^u$ transverse to the Hamiltonian flow by looking at the negative gradient flow lines of $H_2.$
   These  are solutions to the ODE  
      \[
      \dot x_2  =   - (H_2)_{x_2}   = 0  , \quad      \dot y_2  = -  (H_2)_{y_2}   = -\frac{1}{ y_2^3}+1 >0.
      \]
        Therefore, along every negative gradient flow line  the $x_2$-component is constant and the $y_2$-component is strictly increasing. See Figure \ref{fig:ngflK2}. 
   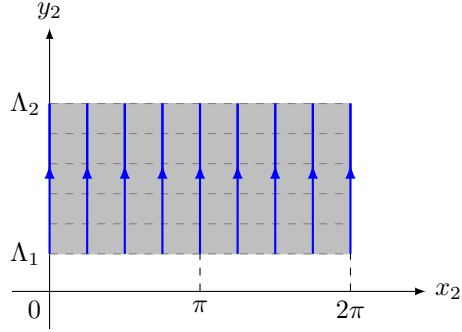
\begin{figure}[ht]
     \centering
\begin{tikzpicture}

 \draw[->] (-0.5,0) to (5,0);
 \draw[->] (0,-0.5) to (0,3.5);

 \filldraw[draw=lightgray,fill=lightgray] (0,0.5) rectangle (4,2.5);
 \draw[dashed] (2,0) to (2,2.5);
 \draw[dashed] (4,0) to (4,0.5);
 \node at ( -0.2,0) [below] {$0$};
 \node at (2,0) [below] {$ \pi$};
 \node at (4,0) [below] {$2\pi$};
 \node at (0,0.5) [left] {$\Lambda_1$};
 \node at (0,2.5) [left] {$\Lambda_2$};
 \node at (0,3.5) [above] {$y_2$};
 \node at (5,0) [right] {$x_2$};

 \draw[dashed, gray] (0,0.5) to (4, 0.5);
 \draw[dashed, gray] (0,0.9) to (4, 0.9);
 \draw[dashed, gray] (0,1.3) to (4, 1.3);
 \draw[dashed, gray] (0,1.7) to (4, 1.7);
 \draw[dashed, gray] (0,2.5) to (4, 2.5);
 \draw[dashed, gray] (0,2.1) to (4,2.1);
 
 \draw[blue, thick, -<] (0.5,2.5) to (0.5, 1.5);
 \draw[blue, thick] (0.5,2.5) to (0.5, 0.5);
 
  \draw[blue,thick, -<] (0 ,2.5) to (0 , 1.5);
 \draw[blue, thick] (0 ,2.5) to (0 , 0.5);
 
  \draw[blue, thick, -<] (1,2.5) to (1, 1.5);
 \draw[blue, thick] (1,2.5) to (1, 0.5);
 
  \draw[blue, thick, -<] (1.5,2.5) to (1.5, 1.5);
 \draw[blue, thick] (1.5,2.5) to (1.5, 0.5);
 
  \draw[blue, thick, -<] (2,2.5) to (2, 1.5);
 \draw[blue, thick] (2,2.5) to (2, 0.5);
 
  \draw[blue, thick, -<] (2.5,2.5) to (2.5, 1.5);
 \draw[blue, thick] (2.5,2.5) to (2.5, 0.5);
 
   \draw[blue,thick, -<] (3.5,2.5) to (3.5, 1.5);
 \draw[blue,thick] (3.5,2.5) to (3.5, 0.5);
 
   \draw[blue,thick, -<] (3,2.5) to (3, 1.5);
 \draw[blue,thick] (3,2.5) to (3, 0.5);
 
   \draw[blue,thick, -<] (4,2.5) to (4, 1.5);
 \draw[blue,thick] (4,2.5) to (4, 0.5);

  \end{tikzpicture}
    \caption{Negative gradient flow lines of $H_2.$ Dashed lines indicate Hamiltonian trajectories.}
    \label{fig:ngflK2}
 \end{figure}

      Pick any $x_2 \in \R / 2 \pi \Z$ and let  $h_{x_2}$ be the corresponding negative gradient flow line of $H_2.$
               Denote by $\Pi \colon \Sigma_{\rm direct}^u \to \R^2$ the projection $(x_1, x_2, y_1, y_2) \mapsto (x_2, y_2).$ Then the preimage $\Pi^{-1}(h_{x_2})$   is an embedded disc in $\Sigma_{\rm direct}^u$ that is transverse to the Hamiltonian flow. 
               The point $\Pi^{-1}(x_2, \Lambda_1)$ is the intersection of the disc $\Pi^{-1}(h_{x_2})$ and the direct circular orbit $\gamma_{\rm direct}^u.$
               By varying $x_2 \in \R / 2 \pi \Z,$ we obtain the desired disc foliation. 
               
Below we shall embed $\Sigma_{\rm direct}^u$ into a weakly convex three-sphere   and construct  a  symmetric finite energy foliation of the three-sphere. We then      obtain   a   disc foliation as the intersection of the projection of the  finite energy foliation   with $\Sigma_{\rm direct}^u.$


       \subsection{Interpolation}\label{sec:interpol}  

 Fix $\Lambda_3 \in (\Lambda_2, +\infty)$ and $0< \varepsilon_0 < \frac{1}{2} ( \Lambda_3 - \Lambda_2).$
   Consider
      \begin{equation*}\label{eq:htilde}
      F( x_2, y_2) = \frac{1}{2} ( y_2 - \Lambda_3)^2 -  \cos \frac{x_2}{2}   +B,
      \end{equation*}
      where
      $(x_2, y_2) \in \R / 4 \pi \Z \times \R_+,$ and
      the constant $B$ satisfies
      \begin{equation}\label{eq:B}
      B<   - \frac{1}{2 ( \Lambda_2+\varepsilon_0)^2} - (\Lambda_2 + \varepsilon_0)  - \frac{1}{2} ( \Lambda_2 +  \varepsilon_0 - \Lambda_3)^2 -1
      \end{equation}
and  define
   \begin{equation}\label{eq:H2tilde}
   K(x_2, y_2) = \left( 1 - f(y_2) \right) H_2(  y_2) + f(y_2) F(x_2, y_2),
   \end{equation}
where $f\colon \R  \to [0,1]$ is a non-decreasing function satisfying  $f(y_2) =0$ for $y_2 <  \Lambda_2 + \varepsilon_0$ and $ f(y_2)=1$ for $y_2 >   \Lambda_2 +2\varepsilon_0.$   
 See Figure \ref{fig3}.

Set
\begin{equation*}\label{eq:barHH}
\bar{H}(x_1, x_2, y_1, y_2) = H_1(x_1, y_1) + K(x_2, y_2) ,
\end{equation*}
where $(x_2, y_2) \in \R/4\pi \Z \times \R_+,$
and denote by $\bar \Sigma$ the connected component of $\bar{H}^{-1}(c)$ that contains a double cover $\bar{\Sigma}_{\rm direct}^u$ of $\Sigma_{\rm direct}^u.$
 Note that $\bar H$ is invariant under the anti-symplectic involution  
\begin{equation}\label{eq:rhofirst}
\rho( x_1, x_2, y_1, y_2) = ( -x_1,  4 \pi - x_2, y_1 , y_2).
\end{equation}

      We claim that for $y_2 \geq \Lambda_1,$ the smooth function $K$ has precisely two critical points $(0, \Lambda_3)$ that is a saddle  and $(2\pi, \Lambda_3)$ that is a global minimum. In fact, we show that 
\begin{equation}\label{eq:dkdy}
K_{y_2} <0 , \ \forall y_2 < \Lambda_3  \quad \text{ and }  \quad  K_{y_2}>0, \ \forall y_2 > \Lambda_3.
\end{equation}
 It suffices to show that $K$ is strictly decreasing in $y_2,$ provided  $y_2 \in [\Lambda_2+\varepsilon_0, \Lambda_2 +2\varepsilon_0].$ We assume $y_2 \in [\Lambda_2+\varepsilon_0, \Lambda_2 +2\varepsilon_0]$ and  find that
 \begin{align*}
 K_{y_2} &= f' \cdot \left( \frac{1}{2 y_2^2 } + y_2+ \frac{1}{2} ( y_2 - \Lambda_3)^2 - \cos \frac{x_2}{2} + B\right) \\
  & \quad + (1-f) \cdot \left( \frac{1}{y_2^3}-1\right)  + f  \cdot  ( y_2 - \Lambda_3   ).
 \end{align*}
 The first term in the right-hand side   is non-positive due to the choice of the constant $B.$ See \eqref{eq:B}.  
 The second term  is    non-positive since $1<\Lambda_2,$ and the last term is also non-positive   because of the choice of $\varepsilon_0.$ 
 Note that at least one of the three terms is non-zero. This implies that   $K$ admits a critical point only if $y_2 = \Lambda_3,$ provided $y_2  \geq \Lambda_1.$  The rest of the claim is straightforward.

  Given $x_2  ,$ define
  \[
  \Lambda_{\max} = \Lambda_{\max}(x_2):= \max \{ y_2 \mid (x_1, x_2, y_1, y_2) \in  \bar{\Sigma}     \}.
  \]
  Notice that  it satisfies  $K (x_2, \Lambda_{\max})=c $ and that
\[
  \Lambda_1 =  \min \{ y_2 \mid (x_1, x_2, y_1, y_2) \in  \bar{\Sigma}     \}.
\]  

We now introduce
\[
G(x_2, y_2) = \frac{1}{2} x_2^2 + \frac{1}{2}( y_2 - \Lambda_3)^2.
\]
Given $\varepsilon_1>0,$  choose any non-decreasing smooth function $g\colon \R \to [0,1]$ satisfying $g(x_2) =0$ for $x_2 < - 2\varepsilon_1$ and $g(x_2)=1$ for $x_2 > -\varepsilon_1$ and define
\[
\tilde K(x_2, y_2) = \left( 1 - g(x_2) \right) G(x_2, y_2) + g(x_2) K(x_2, y_2).
\]
We repeat a similar  business for $x_2 \geq 4 \pi$ in a symmetric way (note that the function $K$ is symmetric in $x_2$ with respect to $x_2  =2\pi$). Namely,   we interpolate between the smooth functions $\tilde K$ and  
\[
\tilde G (x_2, y_2) = \frac{1}{2} ( x _2 - 4\pi)^2 + \frac{1}{2} ( y_2 - \Lambda_3)^2  
\]
using the non-increasing function $\tilde g (x_2) := g( 4 \pi - x_2).$
This produces a smooth function $\tilde H_2 (x_2, y_2),$ which is still symmetric in $x_2$ with respect to $x_2 = 2\pi.$
Note that we have broken the periodicity in  $x_2.$   See Figure \ref{fig3}.

 \begin{figure}[t]
  \centering
  \includegraphics[width=1.0\linewidth]{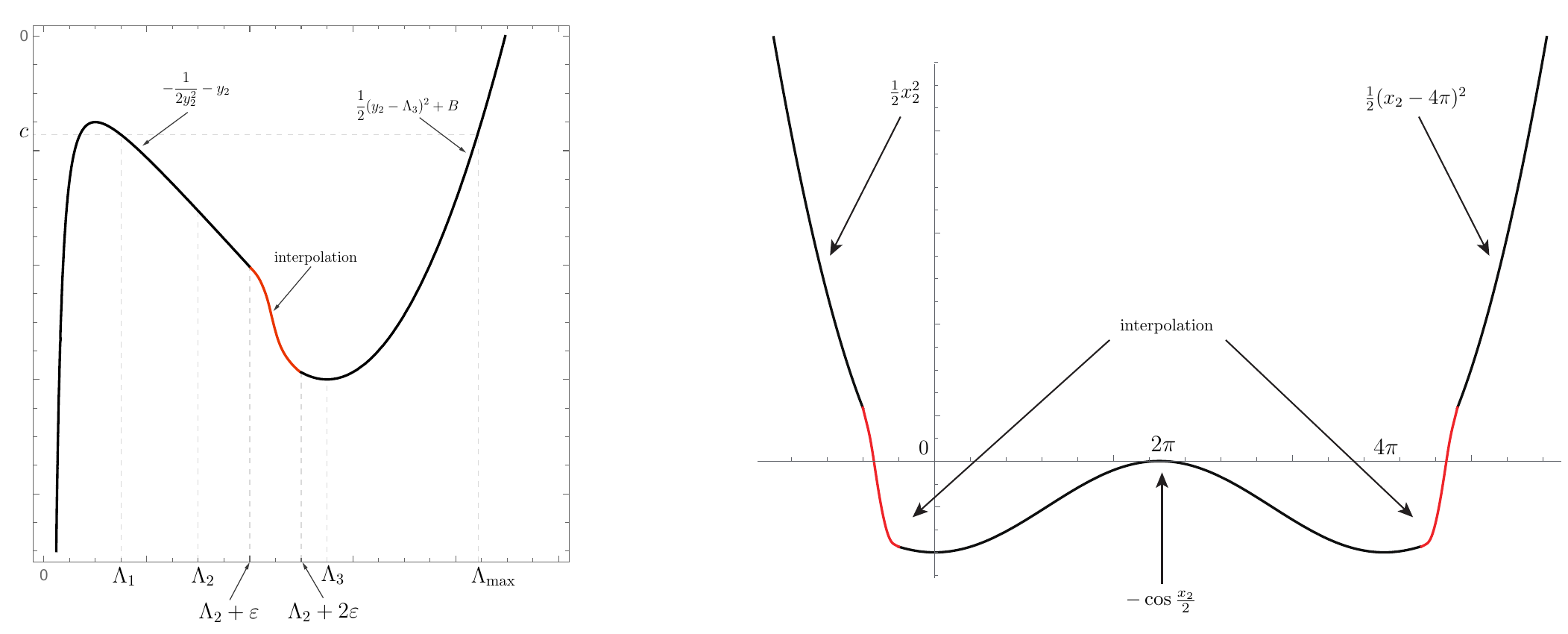}
  \caption{Interpolations for   $y_2$(left) and $x_2$(right)}
  \label{fig3}
\end{figure}

\begin{lemma}\label{lem:criticalpointsofh2tilde}
The function $\tilde H_2$ admits precisely three critical points  $(x_2, y_2) = (0,\Lambda_3), (2\pi, \Lambda_3)$ and $(4\pi, \Lambda_3).$
\end{lemma}
\begin{proof}
By the definition and symmetry of $\tilde H_2,$ it suffices to show that there are no critical points of $\tilde H_2$ in  $\{ -2 \varepsilon_1 < x_2 < - \varepsilon_1 \}. $

We assume $-2 \varepsilon_1 < x_2 < - \varepsilon_1$ and find that 
\begin{align*}
 (\tilde H_2)_{x_2}  &= \left( 1-g(x_2) \right) x_2  + \frac{ f(x_2)g(x_2)}{2} \sin \frac{x_2}{2}+ g'(x_2)(K-G),\\
(\tilde H_2)_{y_2} &= \left( 1-g(x_2)\right)(y_2 - \Lambda_3) + g(x_2) K_{y_2}.
\end{align*}
If $y_2 \geq \Lambda_1,$ then we obtain in view of  \eqref{eq:dkdy}   that
\begin{equation*}\label{eq:dh2dy}
(\tilde H_2)_{y_2} <0,\;  \forall y_2 <\Lambda_3, \quad (\tilde H_2)_{y_2} >0, \; \forall y_2 >\Lambda_3 ,
\end{equation*}
and hence $(\tilde H_2)_{y_2}$ vanishes only if $y_2=\Lambda_3.$
We then compute
\[
K(x_2, \Lambda_3) - G(x_2 , \Lambda_3) =  -\cos \frac{x_2}{2} - \frac{1}{2} x_2^2 + B <0,
\]
so that $( \tilde H_2)_{x_2} |_{ y_2 = \Lambda_3} <0.$

Suppose that $y_2 < \Lambda_1,$ so that
\begin{align*}
 (\tilde H_2)_{x_2}  &= \left( 1-g(x_2) \right) x_2  + g'(x_2)\left( - \frac{1}{2 y_2^2}-y_2 - \frac{1}{2}x_2^2 - \frac{1}{2}(y_2-\Lambda_3)^2  \right),\\
(\tilde H_2)_{y_2} &= \left( 1 - g(x_2) \right) (y_2- \Lambda_3) + g(x_2) \left( \frac{1}{y_2^3}-1 \right).
\end{align*}
Note that $(\tilde H_2)_{y_2}$ may vanish for $y_2 <1.$
For $y_2 >0,$ a global minimum of the function $\frac{1}{2 y_2^2} + y_2 + \frac{1}{2}(y_2-\Lambda_3)^2 $   is positive, and hence we have   $(\tilde H_2)_{x_2}<0,$ provided $y_2<1.$ 



We conclude that  there exist no critical points of $\tilde H_2$  in  $\{ -2 \varepsilon_1 < x_2 < - \varepsilon_1 \}, $  
from which the lemma follows. 
\end{proof}


We define the modified Hamiltonian
\[
\tilde H(x_1, x_2, y_1, y_2) = H_1(x_1, y_2) + \tilde H_2(x_2, y_2) .
\]
The anti-symplectic involution $\rho$ in \eqref{eq:rhofirst} induces an anti-symplectic involution on $\R^3 \times \R_+,$ again denoted by $\rho.$ 
Note that $\tilde H$ is $\rho$-invariant.

The previous lemma implies that   the Hamiltonian $\tilde H$ admits precisely three critical points
\begin{equation}\label{eq:criitodfhildH}
p_{c1} = (0, 0, 0, \Lambda_3), \quad p_{c2} = (0,2\pi,0,\Lambda_3), \quad p_{c3} = (0,4\pi, 0,\Lambda_3).
\end{equation}
Since $\tilde H(p_{c1}) = \tilde H( p_{c3}) = -1+B < \tilde H( p_{c2}) = 1+B<c$ and $p_{c2}$ has Morse index $1$ as a critical point of $\tilde H,$ it follows from standard Morse theory that the energy level $\tilde H^{-1}(c) $ contains a unique connected component, denoted by $\Sigma,$ which is diffeomorphic to $S^3.$
Note that $\Sigma$ is $\rho$-invariant.


For later use, we prove the following assertion.

\begin{lemma}\label{lem"cbar} Along $\Sigma$ we have $y_2>1,$ and hence\[
\tilde H_{y_2}<0, \ \forall y_2 < \Lambda_3, \quad \tilde H_{y_2}>0, \ \forall y_2 > \Lambda_3.
\]
\end{lemma}
\begin{proof} By the definition and  symmetry of $\tilde H_2$, without loss of generality we may assume that $-2 \varepsilon_1 < x_2 < - \varepsilon_1.$ 

Fix $-2 \varepsilon_1 < x_2 < - \varepsilon_1 $ and consider the smooth function 
\[
Q( y_2)= \frac{1}{2}\left( 1 - g(x_2) \right) (y_2- \Lambda_3)^2  + g(x_2) \left(  - \frac{1}{2 y_2^2}-y_2\right), \quad y_2>0.
\]
There is $a=a(x_2)\in (0,1)$ such that the function $Q$ is strictly increasing for $y_2 \in (0,a)$ and is strictly decreasing for $y_2 \in (a, \Lambda_2 + \varepsilon_0).$

Every point $(x_1, x_2, y_1, y_2) \in \Sigma$ with  $y_2 < \Lambda_2 + \varepsilon_0$ satisfies 
\begin{equation*}\label{eq:ywbound}
Q( y_2)=  c - \frac{1}{2}x_2^2 \left( 1 - g(x_2) \right)   - \frac{1}{2}(x_1^2+ y_1^2),
\end{equation*}
and hence we find
\begin{align*}
 Q(y_2)
 &\leq    c - \frac{1}{2}x_2^2 \left(1 - g(x_2)\right) \\
  &<       - \frac{3}{2}g(x_2) +\frac{1}{2} \left( 1 - g(x_2) \right) (1- \Lambda_3)^2 \\
   &=    Q( 1).
\end{align*}
This together with compactness of $\Sigma$ implies that  for any $(x_1, x_2 , y_1, y_2) \in \Sigma$ with   $x_2 \in (-2 \varepsilon_1 ,   - \varepsilon_1)$ the component $y_2$ is bigger than $1.$
  The rest of the lemma is straightforward. This finishes the proof.
\end{proof}

\subsection{Construction of a contact form}  \label{sec:form}



In this section we   construct a $\rho$-invariant Liouville vector field $Y$ that is transverse to $\Sigma .$  Then the   restriction of the one-form $\omega_0(Y, \cdot)$ to $\Sigma$ defines a contact form on $\Sigma$, denoted by $\lambda.$   The anti-symplectic involution $\rho \colon \R^4 \to \R^4$   restricts to an anti-contact involution on $\Sigma,$ still denoted by $\rho.$
 Consequently, we obtain a real contact manifold $(\Sigma, \lambda, \rho).$
Recall from Section \ref{sec:pseudoho} that $\rho$ provides   the exact anti-symplectic involution $\tilde \rho = {\rm Id}_{\R} \times \rho$ on  $( \R \times \Sigma, d(e^r \lambda)).$

We first  claim that  the radial vector field with respect to $p_{c1}$
\[
Y_0 = \frac{1}{2} \left( x_1 \p_{x_1} + x_2 \p_{x_2} + y_1\p_{y_1} + (y_2 - \Lambda_3)\p_{y_2} \right)
\]
is transverse to $\Sigma$ along $\{ x_2 < 2\pi\}.$
By the definition of $\tilde H,$ it suffices to consider the region $\{ - 2 \varepsilon_1 < x_2 < 2 \pi\}.$
We find that
\begin{equation*}\label{eq:varangle}
\begin{aligned}
d\tilde H(Y_0) & = \frac{1}{2} ( x_1^2 +y_1^2) +\frac{1}{2} \left( 1-g(x_2) \right)\left( x_2^2 + (y_2 - \Lambda_3)^2 \right) \\
& \ \ \  + \frac{1}{2}g(x_2) \left(  x_2  K_{x_2} + (y_2 - \Lambda_3) K_{y_2} \right) + \frac{1}{2} x_2 g'(x_2)(K-G). 
\end{aligned}
\end{equation*}
The    first two terms in the right-hand side   are non-negative. The third term is also non-negative because of Lemma \ref{lem"cbar}  and  \eqref{eq:dkdy}. For the last term, we find that
\begin{align*}
& \ \ \ K(x_2, y_2) - G(x_2, y_2)   \\
&= \left( 1 - f(y_2) \right) \left( - \frac{1}{2 y_2^2} - y_2 - \frac{1}{2}(y_2 - \Lambda_3)^2 \right) + f(y_2) \left( B - \cos \frac{x_2}{2} \right) -\frac{1}{2}x_2^2.
\end{align*}
By the choice of the constant $B$ (see \eqref{eq:B}), it is negative. This proves the claim.

 We now consider the Liouville vector field
 \[
 Y_1 = \frac{1}{2} ( x_1 \p_{x_1} + y_1 \p_{y_1})+ \frac{1}{4} \sin \frac{x_2}{2} \p_{x_2} + ( y_2 - \Lambda_3) \left( 1 - \frac{1}{8} \cos \frac{x_2}{2} \right) \p_{y_2}.
 \]
Using \eqref{eq:dkdy} it is straightforward to see that $d\tilde H(Y_1)>0$ along $\Sigma \cap  \{ 0 < x_2 < 4 \pi \}  .$

In order to interpolate between the two Liouville vector fields $Y_0$ and $Y_1,$ we define a smooth function 
\[
\ell(x_2, y_2) = ( y_2 - \Lambda_3) \left( \frac{1}{4} \sin \frac{x_2}{2} - \frac{1}{2}x_2 \right),
\]
so that $X_{\ell} = Y_1 - Y_0.$
Fix $\varepsilon_2>0 $ small enough and pick  any non-decreasing smooth function $h \colon \R \to [0,1]$ satisfying $h(x_2)=0$ for $x_2 < 2\pi - 2 \varepsilon_2$  and $h(x_2) = 1$ for $x_2 > 2 \pi - \varepsilon_2.$ Define
\begin{equation}\label{liouville2iless}
\tilde Y = Y_0 + X_{h\ell  },
\end{equation}
where $X_{h\ell}$ indicates the Hamiltonian vector field associated with $h(x_2)\ell(x_2, y_2).$ 

We claim that the Liouville vector field $\tilde Y$ is transverse to $\Sigma$ along $\{ x_2 < 4 \pi  \}.$ In view  of the argument above, it suffices to consider the region $\{ 2\pi - 2\varepsilon_2 < x_2 < 2 \pi - \varepsilon_2\}. $
We find that
\[
d \tilde H ( \tilde Y ) = (1-h) d  \tilde H(Y_0) + h d   \tilde H(Y_1)+\ell d  \tilde H ( X_h).
\]
The sum of the first two terms in the right-hand side above is positive. 
 For the last term, we find
 \[
 \ell d\tilde H(X_h) = - h'(x_2) \cdot \left( \frac{1}{4} \sin \frac{x_2}{2} - \frac{1}{2} x_2 \right) ( y_2 - \Lambda_3) K_{y_2}.
 \]
This is non-negative because of \eqref{eq:dkdy}. This proves the claim.

We now repeat the same business as above in the region $\{ x_2 > 2\pi\}.$ More precisely, we consider the radial vector field with respect to $p_{c2}$
\[
Y_2 = \frac{1}{2} \left( x_1 \p_{x_1} + (x_2- 4 \pi) \p_{x_2} + y_1\p_{y_1} + (y_2 - \Lambda_3)\p_{y_2} \right)
\]
and interpolate between $\tilde Y$ and $Y_2$ using the non-increasing function 
$\tilde h(x_2) := h( 4 \pi - x_2).$
This provides the   Liouville vector field $Y$ that is transverse to $\Sigma.$ 
By construction, it is invariant under $\rho.$

 \subsection{Periodic orbits}\label{subsec:perlower}
Let $\gamma(t)  = (x_1(t), x_2(t), y_1(t), y_2(t)) $ be a periodic orbit of $\tilde H.$   Denote by 
 \begin{equation*}\label{eq:periodproje}
    \gamma_1(t) = (x_1(t), y_1(t)), \quad \gamma_2(t) = (x_2(t), y_2(t))
    \end{equation*}
    the projections of $\gamma(t)$ to the $(x_1, y_1$)-plane and to the $(x_2, y_2)$-plane, respectively. 
    Then the periodic orbit $\gamma$ belongs to one of the following three categories:  
      \begin{itemize}
      \item $\gamma_1$ is constant; the green curve in Figure \ref{fig:K}.
      
      \item $\gamma_2$ is constant; the red and blue curves in Figure   \ref{fig:K}.
      
      \item  Both $\gamma_1$ and $\gamma_2$ are non-constant; the black curves in Figure \ref{fig:K}.
      
      \end{itemize}

 \begin{figure}[ht!]
  \centering
  \includegraphics[width=1.0\linewidth]{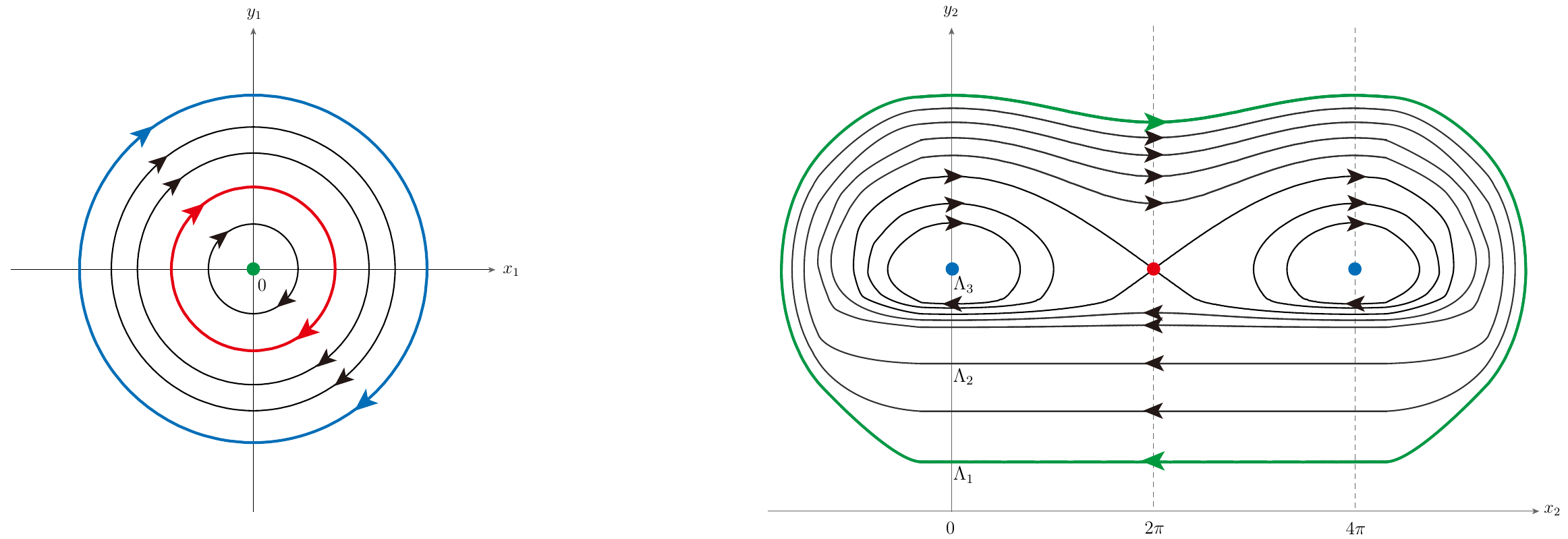}
    \caption{Phase portraits of $K_1$(left) and $K_2$(right)
       }
    \label{fig:K}
\end{figure}

Consider the first case, so that we   have $\gamma_1 \equiv (0,0),$ and hence $\tilde H_2(\gamma_2(t))=c, \forall t.$  
The corresponding periodic orbit will be denoted by $P_0.$ 


We now assume that $\gamma_2$ is constant, so that we have either $\gamma_2 \equiv (0, \Lambda_3),$  $\gamma_2 \equiv (2\pi, \Lambda_3)$ or $\gamma_2 \equiv (4\pi, \Lambda_3).$ See Lemma \ref{lem:criticalpointsofh2tilde}. The corresponding periodic orbits $\gamma$ will be denoted by $P_3 = (w_3, 2\pi), P_2=(w_2, 2\pi)$ and $P_3' = (w_3', 2\pi),$ where
\begin{align*}\label{eqL:binding}
w_3(t)  &= ( r_3 \sin t, 0, r_3 \cos t, \Lambda_3), \\
w_2(t) &= ( r_2 \sin t,2 \pi , r_2 \cos t, \Lambda_3),\\
w_3'(t)  &= ( r_3 \sin t, 4\pi, r_3 \cos t, \Lambda_3), 
\end{align*}
respectively, with
\begin{equation}\label{eq:radii}
  r_2 = \sqrt{ 2 ( c-1-B)}<    r_3 = \sqrt{ 2 ( c + 1 - B)}.
  \end{equation}
  Note that $P_2$ is a symmetric periodic orbit and that $P_3'=(P_3)_{\rho}.$  
    Since 
\[
R = \frac{1}{\lambda(X_{\tilde H})}X_{\tilde H} ,
\]
the Reeb periods of $P_2, P_3$ and $P_3'$ are given by
\[
\tau_2 = \pi r_2^2, \quad  \tau_3 = \tau_3' = \pi r_3^2,
\]
 respectively.

In the last case, i.e.\ both $\gamma_1$ and $\gamma_2$ are non-constant, the periodic orbit $\gamma$ lies in an $S^1$-family of periodic orbits. 

      \subsection{Conley-Zehnder indices}
      
      We shall prove

 \begin{proposition}\label{prop:lowwen}
The contact form $\lambda$ on $\Sigma,$ constructed in Section \ref{sec:form}, is weakly convex. More precisely,
the periodic orbits $P_2,P_3$ and $P_3'$ are non-degenerate and satisfy $\mu_{\rm CZ}(P_2)=2$ and $\mu_{\rm CZ}(P_3)= \mu_{\rm CZ}(P_3')=3.$
The other periodic orbits have the Conley-Zehnder indices at least $3.$ 
In particular, $P_2$ is a unique periodic orbit having $\mu_{\rm CZ}=2.$

 \end{proposition}
\begin{proof} We follow \cite{Salcomm}.

We first consider the periodic orbits  $P_2,P_3$ and $P_3'.$ 
Since $\mu_{\rm CZ}( P_3' ) = \mu_{\rm CZ}(P_3),$ see Section \ref{sec:indexcz},
  we may consider only $P_2$ and $P_3$
along which we have  
      \[
      X_1 = (0,  -  r_i \cos t, 0, - r_i \sin t), \quad X_2 = ( 0, - r_i \sin t, 0, r_i  \cos t), \quad i=2,3
      \]
      where $X_1$ and $X_2$ are as in \eqref{eq:X123}.
As before, we      denote by $\mathfrak{T}$ the trivialisation of  $ ( T \Sigma / \R X_K)   |_{P_i}$ induced by $X_1$ and $X_2.$  
We choose another trivialisation $\Phi$ that is  induced by the canonical basis  
      \[
      V_1 = (0,1,0,0), \quad V_2 = (0,0,0,1).
      \]
In this   trivialisation    the linearised Hamiltonian flow along $P_i$ restricted to  $ ( T \Sigma / \R X_K)   |_{P_i}$  is described by a solution to the ODE
      \[
      \dot{\alpha}_i(s) = A \alpha_i(s), \quad A = \begin{pmatrix} 0 & 1 \\  -\frac{1}{4} \cos \frac{x_2}{2}    & 0 \end{pmatrix}.
      \]
      
      If $i=2,$ then we have $x_2 = 2\pi,$ and hence  
      \[
      A \begin{pmatrix} 1 \\ 0 \end{pmatrix} = \begin{pmatrix}  0 \\ \frac{1}{4} \end{pmatrix}, \quad \quad   A \begin{pmatrix} 0 \\ 1 \end{pmatrix} = \begin{pmatrix}  1 \\ 0 \end{pmatrix}.  
      \]
      This implies that the rotation interval (see \eqref{def:rotin}) of any non-trivial solution to the ODE above contains $0$ as an interior point.
Since  the winding number of the projection  of $X_1$ (and hence also of $X_2$) to the $(x_2,y_2)$-plane are equal to $1,$ it follows that 
      the rotation interval of any non-trivial solution to ODE \eqref{eq:ODEODE} contains $ 2\pi$ as an interior point. We conclude from   the definition of the Conley-Zehnder index    that $P_2$ is non-degenerate and satisfies
      \[
      \mu_{\rm CZ}(P_2) = 2.
      \]


      Consider $P_3,$ so that $x_2 = 0.$ Then a continuous argument of any   non-vanishing solution $\alpha$ to the ODE above is of the form
      \[
      \theta (s) = -\frac{1}{2}s + \theta_0
      \]
for some constant $\theta_0 \in \R.$ Since the Hamiltonian period of $P_3$ equals $2\pi,$ the rotation interval of $\alpha$ 
is contained in $(-2\pi, 0).$
Since  the basis $\{ X_1 , X_2 \}$ has the negative orientation with respect to the   basis $\{ V_1 , V_2\} ,$ and 
the winding number of the projection  of $X_1$   to the $(x_2,y_2)$-plane equals $1,$  as  in the previous case, this implies that 
      the rotation interval of any non-trivial solution to ODE \eqref{eq:ODEODE} is contained in $(2\pi, 4\pi)$ from which we obtain   that $P_3 $ is non-degenerate and satisfies
\[
\mu_{\rm CZ}(P_3)=3.
\]



Let $P=(w,T)$ be a periodic orbit not corresponding to the critical points of $\tilde H_2.$ See   Figure \ref{fig:K}. We write $w(t) = (x_1(t), x_2(t), y_1(t), y_2(t))$ and  denote by  $w_1(t)=(x_1(t), y_1(t))$ and   $w_2(t)=(x_2(t), y_2(t))$ the projections of $w(t)$ to the $(x_1, y_1)$-plane and  $(x_2, y_2)$-plane, respectively.  Note that 
\[
\dot{w}_2(t) = ( \dot{x}_2(t), \dot{y}_2(t) ) = \left( \tilde H_{y_2}(w(t)), - \tilde H_{x_2}( w(t) ) \right).
\]
For $j=1,2,$
  $w_j(t)$ is a    closed curve  which is oriented in the clockwise direction, and hence its winding number with respect to the standard basis is given by
\[
{\rm wind}\left( t \mapsto \dot{w}_j(t) \right) \leq -1,
\]
where the equality holds if and only if $w_j(t)$ is a simple closed curve.

We abbreviate by
\[
Y_1(t) = \left( \tilde H_{y_2} (w(t)), \tilde H_{x_2}(w(t)) \right), \quad Y_2(t) = \left( \tilde H_{x_2}(w(t)), -\tilde H_{y_2}(w(t))\right),
\]
which are the projections of $X_1$ and $X_2$ along $P$ to the $(x_1, y_1)$-plane. 
Since $P$ does not correspond to the critical points of $\tilde H_2,$ these vectors are non-vanishing and linearly independent.
Hence, we may compute the rotation interval associated with the transverse linearised flow of $X_K$ along $P$ using the trivialisation of the $(x_1, y_1)$-plane induced by $Y_1$ and $Y_2.$
We find the winding numbers of $Y_1$ and $Y_2$ with respect to the standard basis  
\begin{align}\label{eq:windofY1}
     \begin{split}{\rm wind}\left(t \mapsto Y_1(t)\right) &= {\rm wind}\left(t \mapsto Y_2(t)\right) \\
&= {\rm wind}\left( t \mapsto ( \tilde H_{y_2} ( w(t)), \tilde H_{x_2}(w(t)))\right)\\
&= - {\rm wind}\left( t \mapsto ( \tilde H_{y_2} ( w(t)), -\tilde H_{x_2}(w(t)))\right)\\
&= - {\rm wind} \left( t \mapsto \dot{w}_2(t) \right) \\
&  \geq 1.
 \end{split}
     \end{align}
 
Suppose that $P$ corresponds to a black curve in Figure \ref{fig:K}, so that $\dot{w}_1(t)$ is non-vanishing.   Since $\dot{w}_1(t)$ is preserved by the linearised flow projected to the $(x_1, y_1)$-plane, there is a solution $\alpha(t)$ to ODE \eqref{eq:ODEODE} whose projection to the $(x_1, y_1)$-plane equals $\dot{w}_1(t).$ We denote such a projection by $\bar{\alpha}(t)$
 and then compute its winding number with respect to the basis $\mathfrak{B} = \{ Y_1(t), Y_2(t)\}$
 \begin{align*}
 {\rm wind}\left( t \mapsto \bar{\alpha}(t) ; \mathfrak{B} \right) & =  {\rm wind}\left( t \mapsto Y_1(t) \right) - {\rm wind}\left( t \mapsto \bar{\alpha}(t) \right) \\
 & =  {\rm wind}\left( t \mapsto Y_1(t) \right) - {\rm wind}\left( t \mapsto  \dot{w}_1(t) \right) \\
& \geq 2,
 \end{align*}
where we have used the fact that the basis $\mathfrak{B}$ has negative orientation with respect to the standard basis. This implies that the rotation interval of $P$ contains $4\pi,$ from which we conclude that $\mu_{\rm CZ}(P) \geq 3.$

We now assume that $P$ corresponds to the green curve in Figure \ref{fig:K}, so that $w_1(t) \equiv (0,0). $ In this case, the linearised Hamiltonian flow along $P$ restricted to the $(x_1, y_1)$-plane   is described by a solution to the ODE
     \begin{equation*}\label{eq:ODE1}
       \dot{\beta} = \begin{pmatrix} 0 & 1 \\ -1 & 0 \end{pmatrix} \beta .
       \end{equation*}
As in the case of $P_3,$ this implies that the rotation interval of $\beta$ is contained in $(-\infty, 0), $  and hence we find using \eqref{eq:windofY1} that        the rotation interval of any non-trivial solution to ODE \eqref{eq:ODEODE} is contained in $(2\pi,+\infty).$ Consequently, we have $\mu_{\rm CZ}(P) \geq 3.$

 This finishes the proof.
\end{proof}

        \subsection{Construction of a finite energy foliation}\label{sec:lowerconstruction}
        
       Let $\overline{X}_1 , \overline{X}_2$ be the two vector fields that span the contact structure $\xi = \ker \lambda$,  see \eqref{eq:trixi}.  
        We define the $d\lambda$-compatible almost complex structure $J \colon \xi \to \xi$ by 
      \begin{equation}\label{eq:Jcomp}
      J ( \overline X_1 ) =  \overline X_2.
      \end{equation}
    Note that it is $\rho$-anti-invariant. See \eqref{eq:Jrhoanti}.

 Abbreviate $S^1 = \R /\Z.$ Let   $\tilde u = (a,u) \colon \R \times S^1 \to  \R \times \Sigma$ be a $\tilde J$-holomorphic curve, where $\tilde J$ is given as in \eqref{eq:SFTJ}. Being $\tilde J$-holomorphic, $\tilde u$ satisfies
      \begin{equation}\label{eq:jhols}
      \begin{cases}
      \pi u_s + J(u) \pi u_t= 0,\\
      \lambda( u_t) = a_s,\\
      \lambda( u_s)= -a_t,
      \end{cases}
      \end{equation} 
      where $\pi \colon T\Sigma \to \xi $ is the projection along $R.$  
      We make the following Ansatz:
 
 \smallskip
 
 \begin{quote}
     {\bf Ansatz.} The $\Sigma$-part $u$ of $\tilde u$ takes the form
     \begin{equation}\label{eq:anstz}
      u(s,t) = ( r(s) \sin  2\pi t, x_2(s), r(s) \cos 2 \pi t, y_2(s))
      \end{equation}
    \quad \quad \quad  \quad \   with $r(s) \neq 0$ and $\dot r (s) \neq0.$
 \end{quote}

 We only consider the case  that  we have   $x_2  \equiv 0, x_2\equiv 2 \pi, x_2 \equiv 4 \pi$ or $y_2 \equiv \Lambda_3 $ with $0 \leq x_2(s) \leq 4 \pi, \forall s.$ Recall that in this case we have $\tilde H_2 = K,$ where the latter is as in \eqref{eq:H2tilde}.    
Since $u(s,t) \in \Sigma ,$ this implies that 
      \begin{equation}\label{eq:energy}
      \frac{1}{2} r(s)^2 + K( x_2(s), y_2(s) ) = c.
      \end{equation}
 
Following the argument given in \cite[Chapter 5]{dPS1} we shall show that the corresponding $\tilde u$ is a finite energy plane asymptotic to $P_2, P_3$ or $P_3'$ or  a finite energy cylinder asymptotic to $P_3$ or $P_3'$ at its positive puncture and to $P_2$ at its negative puncture.

We first assume the case where $x_2(s) \in [0, 2 \pi],\forall s,$ so that the Liouville vector field is given by $Y = Y_0 + X_{h \ell},$ see \eqref{liouville2iless}, and hence the contact form $\lambda$ equals 
\begin{align*}
\lambda &= \frac{1}{2} \left( y_1 dx_1 + (y_2 - \Lambda_3)dx_2 - x_1 dy_1 - x_2 dy_2 \right) - d(h\ell)\\
&= \frac{1}{2}  y_1 dx_1 - \frac{1}{2}x_1 dy_1 - \frac{1}{2}  (  x_2 +  h(x_2)   (\frac{1}{2}   \sin \frac{x_2}{2} -  x_2  )   ) dy_2  \\
& \ \ \ + \frac{1}{2} (y_2 - \Lambda_3) (    1 - h'(x_2) ( \frac{1}{2}\sin\frac{x_2}{2} - x_2  ) - h(x_2)  ( \frac{1}{4}\cos\frac{x_2}{2} -1 )           ) dx_2. 
 \end{align*}

 We compute
 \begin{align*}
 \nonumber  \lambda(u_t)&=      \pi r^2,\\
\label{eqlambdaus} \lambda(u_s)&=  - \frac{1}{2}  (  x_2 +  h(x_2)   (\frac{1}{2}   \sin \frac{x_2}{2} -  x_2 ))\dot y_2  \\
& \ \ \ + \frac{1}{2} (y_2 - \Lambda_3)(    1 - h'(x_2) ( \frac{1}{2}\sin\frac{x_2}{2} - x_2  ) - h(x_2)  ( \frac{1}{4}\cos\frac{x_2}{2} -1 )           ) \dot x_2.
 \end{align*}
Since $x_2 \equiv 0,$ $x_2 \equiv 2 \pi $ or $y_2 \equiv \Lambda_3,$ we obtain 
\[
\lambda(u_s)=0.
\]
This together with  the last identity of \eqref{eq:jhols} implies that $a$ is independent of $t$, so 
 we may  choose 
            \begin{equation}\label{eq:ast}
            a(s,t) = a(s) = \int_0^s \pi r( u)^2 d u 
            \end{equation}
        from which we see that $a \to \pm \infty$ as $ s \to \pm \infty.$

   The first identity  of \eqref{eq:jhols} implies that $\pi u_s$ and $\pi u_t$ are linearly independent, so that there exist $A, B, C$ and   $D$ such that
 \[  
 \overline X_1 = A \pi u_s + B \pi u_t  \quad \text{ and } \quad   \overline X_2 = C \pi u_s + D \pi u_t.
 \]
 By definition of $J,$ we have $A= D$ and $B=-C,$  and hence  
\begin{equation}\label{eq:AB}
   \overline X_1 = A \pi u_s + B \pi u_t  \quad \text{ and } \quad  
   \overline X_2  = -B \pi u_s + A \pi u_t .
  \end{equation}

Let  $P \colon \R^4 \to \R^2$ be the projection $(x_1, x_2, y_1, y_2) \mapsto (x_2, y_2)$ and denote 
\[
P( \overline X_1)= \begin{pmatrix} \square_1 \\ \square_2 \end{pmatrix}, \quad P(\overline X_2) = \begin{pmatrix} \triangle_1 \\ \triangle_2 \end{pmatrix}, \quad P( \pi u_s) =\begin{pmatrix} a_1 \\ a_2 \end{pmatrix}, \quad P(\pi u_t) = \begin{pmatrix} b_1 \\ b_2 \end{pmatrix}.
\]
A direct computation using \eqref{eq:energy} shows that 
\[
a_1 b_2 - a_2 b_1 =   - \frac{ \pi r^3\dot r }{k}  \neq0,
\]
where
\begin{align*}
      k  &= \lambda(X_{\tilde H})\\  &=  \frac{r^2}{2}   + \frac{1}{2} (y_2 - \Lambda_3) (    1 - h'(x_2) ( \frac{1}{2}\sin\frac{x_2}{2} - x_2  ) - h(x_2)  ( \frac{1}{4}\cos\frac{x_2}{2} -1 )           ) K_{y_2} \\ 
& \ \ \  + \frac{1}{2}  (  x_2 +  h(x_2)   (\frac{1}{2}   \sin \frac{x_2}{2} -  x_2  )   ) K_{x_2}.
\end{align*}
By means of \eqref{eq:AB}   we obtain  
   \[  
      \frac{1}{ a_1 b_2 - a_2 b_1} \begin{pmatrix} b_2 & - b_1 \\ - a_2 & a_1 \end{pmatrix} \begin{pmatrix} \square_1 \\ \square_2 \end{pmatrix} = 
   \begin{pmatrix} A \\ B \end{pmatrix} = \frac{1}{ a_1 b_2 - a_2 b_1} \begin{pmatrix} -a_2 & a_1 \\ - b_2 & b_1 \end{pmatrix} \begin{pmatrix} \triangle_1 \\ \triangle_2 \end{pmatrix},
\]  
and hence
 \begin{align}\label{1st} 
     \begin{split}
    b_2 \square_1 - b_1 \square_2 &= -a_2 \triangle_1 + a_1 \triangle_2,\\
  -a_2 \square_1 + a_1 \square_2 &= -b_2 \triangle_1 + b_1 \triangle_2.
     \end{split}
     \end{align}
    We compute that  
         \[
     \triangle _1 \square_2 - \triangle_2 \square_1 = \frac{ r^2 ( r^2 +  K _{x_2}   ^2 + K_{y_2}  ^2)}{2k } > 0.
     \]
Then  solving \eqref{1st}  provides
     \begin{align}\label{eq:finalsol}
     \begin{split}
     \begin{pmatrix} \dot x_2  \\ \dot y_2  \end{pmatrix} &= \frac{1}{ \triangle _1 \square_2 - \triangle_2 \square_1 } \begin{pmatrix} - \square_1 & \triangle_1 \\  - \square_2 & \triangle_2 \end{pmatrix} \begin{pmatrix} b_2 \square_1 - b_1 \square_2 \\ -b_2 \triangle_1 + b_1\triangle_2 \end{pmatrix} \\
     &= \frac{1}{ \triangle _1 \square_2 - \triangle_2 \square_1 } \begin{pmatrix}  -b_2 ( \triangle_1^2 + \square_1^2) + b_1 ( \triangle_1 \triangle_2 + \square_1 \square_2) \\ b_1 ( \triangle_2^2 + \square_2^2) -b_2 ( \triangle_1 \triangle_2 + \square_1 \square_2) \end{pmatrix}.
     \end{split}
     \end{align}

      Suppose first that $x_2 \equiv 0.$  Using \eqref{eq:energy} and \eqref{eq:finalsol} we find   
\begin{equation}\label{y21}
\dot y_2   (s)  = \frac{     - 2 \pi r(s)^2    K_{y_2} (0, y_2(s))  }{ (   K_{y_2} (0, y_2(s)) )^2 + r(s)^2  } =  \frac{     - 4 \pi (c - K (0, y_2(s)) )    K_{y_2} (0, y_2(s)) }{ (  K_{y_2} (0, y_2(s)) )^2 + 2 (c - K (0, y_2(s)) )  }   ,
\end{equation} 
which may be seen as a differential equation of the type 
\begin{equation}\label{eq:ODEtype}
\dot y_2 (s) = P( y_2(s)).
\end{equation}
Here, $P=P(y_2)$  is a smooth function defined on the interval $[y_2^-, y_2^+],$ where  $(y_2^-, y_2^+ ) = (\Lambda_1, \Lambda_3)$ or $(y_2^-, y_2^+ ) = (\Lambda_3, \Lambda_{\max}).$ 
It follows  from   \eqref{eq:dkdy} that the function $P$ satisfies the following properties:
      \begin{align}\label{eq:properP}
     \begin{split}
 \begin{cases}
P(y_2) = 0 & \quad  y_2 = \Lambda_1, \Lambda_3, \Lambda_{\max} ,\\
 P(y_2)>0  &  \quad  y_2 \in (\Lambda_1, \Lambda_3),\\
  P(y_2)<0 & \quad   y_2 \in (\Lambda_3, \Lambda_{\max}). 
 \end{cases}
      \end{split}
     \end{align}
Thus, any solution $y_2=y_2(s)$ to ODE  \eqref{y21} with initial condition $y_2(0) \in (\Lambda_1, \Lambda_3)$ is strictly increasing and satisfies
\[
\lim_{s \to -\infty} y_2(s) =\Lambda_1 \quad \text{ and } \quad \lim_{s \to +\infty} y_2(s) = \Lambda_3,
\]
and any solution with      $y_2(0) \in (\Lambda_3, \Lambda_{\max})$ is strictly decreasing and satisfies
\[
\lim_{s \to -\infty} y_2(s) =\Lambda_{\max} \quad \text{ and } \quad \lim_{s \to +\infty} y_2(s) = \Lambda_3.
\]
By construction such a solution   yields a solution
\begin{equation}\label{eq:x20case}
u(s,t) = ( r(s) \sin 2 \pi t, 0, r(s) \cos 2 \pi t, y_2(s)), \; (s,t) \in \R \times S^1
\end{equation}
  to the first equation in \eqref{eq:jhols}. Here, the radius $r(s)$ is determined by  relation \eqref{eq:energy} with $x_2(s) = 0$ and satisfies 
\[
\lim_{s \to -\infty} r(s) = 0, \quad \lim_{s \to +\infty} r (s) = r_3 ,
\]
where $r_3 $ is as in \eqref{eq:radii}.
  This together with $a(s,t),$ defined in \eqref{eq:ast}, produces a $\tilde J$-holomorphic curve $\tilde u = (a, u ) \colon \R \times S^1 \to \R \times \Sigma.$  Note that this solution depends on the choice of the initial condition $y_2 (0) \in (\Lambda_1, \Lambda_3) \cup (\Lambda_3 , \Lambda_{\max}).$ 
 For any $N>0$ and for any $\phi \colon \R \to [0,1]$ with $\phi' \geq0,$ we have
\[
\int_{[-N,N] \times S^1} d \tilde{u} ^*   ( \phi \lambda) = \phi( a (N)  ) \pi r (N)^2 - \phi( a (-N)) \pi r (-N)^2,
\]
and hence 
\[
E(\tilde u ) =\tau_3= \pi r_3^2.
\]
  Moreover, the  mass of $\tilde u $ at $-\infty$ is given by
\[
m(-\infty)= \lim_{N \to -\infty}  \int_{\{ N \} \times S^1} u ^* \lambda = \lim_{ N \to -\infty} \pi r (N)^2 = 0 ,  
\]
implying that   $-\infty$ is a removable puncture of $\tilde u$. See \cite[Section 1]{HWZII}.
After removing it, we obtain an embedded finite energy $\tilde J$-holomorphic plane $\tilde u = (a, u) \colon \C \to \R \times \Sigma $ asymptotic to $P_3$ at its  positive puncture $s=+\infty.$

 In the case $x_2 \equiv 2 \pi, $  we obtain  
 \begin{equation}\label{y22} 
\dot y_2   (s)  =  \frac{     - 4 \pi (c - K (2 \pi, y_2(s)) )    K_{y_2}(2 \pi, y_2(s))}{  (   K_{y_2}(2 \pi, y_2(s)) )^2 + 2 (c - K (2 \pi, y_2(s)) )  }   ,
\end{equation}
which may be seen as a differential equation satisfying the same properties as above. See \eqref{eq:ODEtype} and  \eqref{eq:properP}. 
Note that the radius $r(s)$,  determined by   \eqref{eq:energy} with $x_2(s) = 2\pi$,   satisfies 
\[
\lim_{s \to -\infty} r(s) = 0, \quad \lim_{s \to +\infty} r (s) = r_2 ,
\]
where $r_2 $ is as in \eqref{eq:radii}.
Then arguing in a similar manner  we find an embedded finite energy $\tilde J$-holomorphic plane $\tilde u$, depending on the    initial condition $y_2 (0) \in (\Lambda_1, \Lambda_3) \cup (\Lambda_3 , \Lambda_{\max}),$   asymptotic to $P_2$ at its  positive puncture $s=+\infty$ and having energy equal to $\tau_2=\pi r_2^2.$
       Note that it is invariant, i.e.\   $\tilde u = \tilde u_{\rho}$. See Section \ref{sec:pseudoho}.

  Finally, assume that    $  y_2 \equiv \Lambda_3$ with $0 \leq x_2(s) \leq 2\pi, \forall s,$   so that 
\begin{equation}\label{x21}
\dot x_2   (s) 
= \frac{ - 4 \pi ( c - K  ( x_2(s), \Lambda_3) )K_{x_2}(x_2(s), \Lambda_3)}{ (K_{x_2}(x_2(s), \Lambda_3))^2 +2(c - K (x_2(s), \Lambda_3))}, 
\end{equation}
which may be seen as a differential equation of the type
\[
\dot x_2 (s) = Q(x_2(s)).
\]
Here, $Q = Q(x_2)$ is a smooth function defined on $[0, 2\pi ]. $   Notice that
\[
\begin{cases}
Q(x_2) =0 & \quad x_2 = 0, 2\pi,  \\
Q(x_2)<0 & \quad   0 <x_2 < 2\pi , 
\end{cases}
\]
from which we see that any solution $x_2 = x_2(s)$ to ODE \eqref{x21} with initial condition $x_2(0) \in (0, 2\pi)$ is strictly decreasing and satisfies
\[
\lim_{s \to -\infty} x_2(s) = 2\pi \quad \text{ and } \quad \lim_{s \to +\infty} x_2(s) = 0.
\]
In view of the construction such a  solution  gives rise to a solution
\begin{equation}\label{eq:y2lambda3case}
u(s,t) = ( r(s) \sin 2 \pi t, x_2(s), r(s) \cos 2 \pi t, \Lambda_3), \ (s,t) \in \R \times S^1
\end{equation}
to the first equation in \eqref{eq:jhols}. The radius $r(s)$ is determined by   \eqref{eq:energy} with $y_2(s) = \Lambda_3$ and satisfies 
\[
\lim_{s \to -\infty} r(s) = r_2 \quad \text{ and } \quad \lim_{s \to +\infty} r (s) = r_3 ,
\]
where $r_2 $ and $r_3$ are  as in \eqref{eq:radii}.
  This together with $a(s,t),$ defined in \eqref{eq:ast}, produces a $\tilde J$-holomorphic curve $\tilde u = (a, u ) \colon \R \times S^1 \to \R \times \Sigma,$ depending on   the choice of the initial condition $x_2 (0) \in (0, 2 \pi). $  
The Hofer energy is given by $\tau_3$ and the mass of $\tilde u$ at $-\infty$ is equal to $  \tau_2=\pi r_2^2 ,$  implying that  $-\infty$ is non-removable. 
     Thus,   we obtain an embedded finite energy $\tilde J$-holomorphic cylinder $\tilde v = (b, v) \colon \R \times S^1 \to \R \times \Sigma $ asymptotic to $P_3$ at its positive puncture $s=+\infty$ and to $P_2$ at its  negative puncture $s=-\infty.$

We   now   consider the case   $x_2 \equiv 4\pi  $ or $y_2 \equiv \Lambda_3$ with $2\pi \leq x_2(s) \leq 4\pi, \forall s.$ Instead of the direct construction as above, we shall use the symmetry induced by the anti-contact involution $\rho.$


Let $\tilde u = (a,u) \colon \C \to \R \times \Sigma,$ where $u(s,t)$ is as in \eqref{eq:x20case},  
be a finite energy $\tilde J$-holomorphic plane,  corresponding to $x_2 \equiv 0.$  It is asymptotic to $P_3$ at its positive puncture.
The discussion above shows that $ \tilde u_{\rho}(s,t)  =  \left( a_{\rho}(s,t) ,   u_{\rho}(s, t) )\right),$ 
where 
\begin{align*}
a_{\rho}(s,t)  = a(s,-t) \quad \text{ and }  \quad u_{\rho}(s,t)  = (  r(s) \sin 2 \pi t, 4\pi, r(s) \cos 2 \pi t, y_2(s)   ),
\end{align*}
is a finite energy $\tilde J$-holomorphic plane asymptotic to $P_3' =  (P_3)_{\rho}.$
As above,     if  $y_2(0) \in (\Lambda_1, \Lambda_3),$ then $y_2(s)$  is strictly increasing and satisfies
\[
\lim_{s \to -\infty} y_2(s) =\Lambda_1 \quad \text{ and } \quad \lim_{s \to +\infty} y_2(s) = \Lambda_3,
\]
and if        $y_2(0) \in (\Lambda_3, \Lambda_{\max}),$ then $y_2(s)$ is strictly decreasing and satisfies
\[
\lim_{s \to -\infty} y_2(s) =\Lambda_{\max} \quad \text{ and } \quad \lim_{s \to +\infty} y_2(s) = \Lambda_3.
\]

In a similar manner,  if $\tilde u = (a,u) \colon \R \times S^1 \to \R \times \Sigma,$ where   $u(s,t)$ is as in \eqref{eq:y2lambda3case}   and $0\leq x_2(s) \leq 2 \pi, \forall s,$ 
is a finite energy $\tilde J$-holomorphic cylinder,  then  $ \tilde u_{\rho}(s,t) =  \left( a_{\rho}(s,t) ,   u_{\rho}(s, t) )\right),$ where $a_{\rho}$ is as above  and 
\[
u_{\rho}(s,t)  = (  r(s) \sin 2 \pi t, 4\pi-  x_2(s) , r(s) \cos 2 \pi t, \Lambda_3  ),
\]
is a finite energy $\tilde J$-holomorphic cylinder, asymptotic to $P_3'  =  (P_3)_{\rho}$ at  its positive puncture and to $P_2 = (P_2)_\rho$ at its negative puncture. Moreover, if  $    4 \pi - x_2(s)$ has the initial condition    $4\pi - x_2(0) \in ( 2\pi, 4 \pi),$ then it is strictly increasing and satisfies
\[
\lim_{s \to -\infty}    4 \pi - x_2(s) = 2\pi \quad \text{ and } \quad \lim_{s \to +\infty}    4 \pi - x_2(s) = 4\pi.
\]

We have constructed the following embedded finite energy $\tilde J$-holomorphic curves:
\begin{itemize}

\item  finite energy planes $\tilde u_{3,1}$ and $\tilde u_{3,2},$ which corresponds to $x_2 \equiv 0$ and initial conditions $y_2(0) \in (\Lambda_1, \Lambda_3)$ and $y_2(0) \in (\Lambda_3, \Lambda_{\max}),$ respectively.
They are both asymptotic to $P_3.$

\item  finite energy planes $\tilde u_{2,1}$ and $\tilde u_{2,2},$ corresponding to $x_2 \equiv 2 \pi$ and initial conditions $y_2(0) \in (\Lambda_1, \Lambda_3)$ and  $y_2(0) \in (\Lambda_3, \Lambda_{\max}),$ respectively. 
They are invariant and     asymptotic to $P_2.$

\item  finite energy planes $\tilde u_{3,1}' = (\tilde u_{3,1})_\rho$ and $\tilde u_{3,2}'=(\tilde u_{3,2})_\rho,$ which corresponds to $x_2 \equiv  4\pi $ and initial conditions $y_2(0) \in (\Lambda_1, \Lambda_3)$ and $y_2(0) \in (\Lambda_3, \Lambda_{\max}),$ respectively. 
They are both asymptotic to $P_3'= (P_3)_\rho.$

\item  finite energy cylinders $\tilde v_1$ and $ \tilde v_2 = (\tilde v_1)_\rho,$ corresponding to $y_2 \equiv \Lambda_3 $ and initial conditions $x_2(0) \in (0, 2\pi)$ and $x_2(0) \in (2 \pi, 4 \pi),$ respectively.  
The cylinder $\tilde v_1$ is    asymptotic to $P_3$ at  $+\infty$ and to $P_2$ at   $-\infty,$ and $\tilde v_2$ is   asymptotic to $P_3'=(P_3)_\rho$ at  $+\infty$ and to $P_2=(P_2)_\rho$ at   $-\infty.$

\end{itemize}

       The following statement is  taken from \cite[Theorem 1.5]{HWZIII} and \cite[Theorem 4.5.44]{Wendl1}. Note that $J$ is not required to be generic.        
       \begin{theorem}[{Hofer-Wysocki-Zehnder, Wendl}]\label{HWZThm}
      Let a closed three-manifold $M$ be equipped with a contact form $\lambda$ and a $d\lambda$-compatible almost complex structure $J.$   Assume that $\tilde u = (a, u) \colon \C \to \R \times M$ is an embedded finite energy $\tilde J$-holomorphic plane, asymptotic to a non-degenerate periodic orbit $P=(x,T) $ having $\mu_{\rm CZ}(P)=3.$ Then  there exists an $\delta>0$ and an embedding
      \[
      \tilde \Phi  \colon \R \times (-\delta, \delta) \times \C \to \R \times M, \ (\sigma, \tau, z) \mapsto ( a_\tau (z) + \sigma, u_\tau (z)),
      \]
      having the following properties:
      \begin{enumerate}
      \item For each $\sigma \in \R$ and $\tau \in (-\delta, \delta)$ the map $\tilde u_{\sigma, \tau}= \tilde \Phi( \sigma, \tau, \cdot)$ is up to parametrisation an embedded finite energy $\tilde J$-holomorphic plane and $\tilde u_{0,0} = \tilde u.$
      
      \item  The map $u_{\cdot}(\cdot) \colon (-\delta, \delta) \times \C \to M$ is an embedding and its image is disjoint from the asymptotic limit $P.$ In particular, the maps $u_{\tau} \colon \C \to M$ are embeddings for all $ \tau \in (-\delta, \delta),$
      with mutually disjoint images that do not intersect their asymptotic limit $P.$

      \item For any sequence $\tilde u_k = (a_k, u_k)$ of finite energy $\tilde J$-holomorphic planes such that they are all asymptotic to $P$ and satisfy $\tilde u_k \to \tilde u$ in $C^{\infty}_{\rm loc}(\C)$  for all $k$ sufficiently large, there exists   a sequence $(\sigma_k, \tau_k) \in \R \times (-\delta, \delta)$ converging to $(0,0)$ and a sequence $\varphi_k \colon \C \to \C$ of biholomorphic reparametrisations of $\C$ of the form $z \mapsto bz+c$ with constants $b, c \in \C,$ such that  $\tilde u_k = \tilde u_{\sigma_k, \tau_k} \circ \varphi_k.$ 
            \end{enumerate}
       \end{theorem}

       An application of the previous theorem to the finite energy plane $\tilde u_{3,1}$ 
       (recall that $P_3$ is non-degenerate and satisfies $\mu_{\rm CZ}(P_3)=3,$ see Proposition \ref{prop:lowwen}) 
       yields a maximal one-parameter family of embedded finite energy $\tilde J$-holomorphic planes
       \[
       \tilde u_{\tau} = (a_\tau, u_\tau) \colon \C  \to \R \times \Sigma, \quad \tau \in (\tau_-, \tau_+),
       \]
       all asymptotic to $P_3.$         
       This family 
       is non-compact. Indeed, otherwise there is $\tau_1 \neq \tau_2$ with $u_{\tau_1}(\C) \cap u_{\tau_2}(\C) \neq \emptyset.$ Then it follows from \cite[Theorem 1.4]{HWZII} that $u_{\tau_1}(\C) = u_{\tau_2}(\C),$ and hence we find an   open book decomposition of $\Sigma$ with binding $P_3$ and disc-like pages.  By the usual argument, every page is transverse to the Reeb vector field, implying that every periodic orbit other than $P_3$ is linked to $P_3.$ This contradicts to the presence of $P_2$ and $P_3'.$

       We now take a look at the breaking of the family $\{ \tilde u_\tau\}$ as $\tau \to \tau_\pm. $        For simplicity we write $\tau_-=0$ and $\tau_+=1.$  
 We assume that $\tau$ strictly increases in the direction of the Reeb vector field.

       Pick a sequence $\tau_n \to 0^+$ as $n \to +\infty$ and write $\tilde u_n:= \tilde u_{\tau_n}, n \in \N.$ 
              By a standard argument (see \cite{dPS1, HWZ03foliation, HS11}), we find that $\tilde u_n$ converges to a holomorphic building with height $2.$ The top consists of an embedded finite energy $\tilde J$-holomorphic cylinder $\tilde v=(b,v) \colon \C \setminus \{ 0 \} \to \R \times \Sigma $ asymptotic to $P_3$ at its positive puncture $+\infty$ and to an index $2$ orbit $Q$ at at its negative puncture $0.$            
              The bottom consists of an embedded finite energy $\tilde J$-holomorphic plane $\tilde u=(a,u) \colon \C \to \R \times \Sigma$ asymptotic to $Q.$ Moreover,   given a neighbourhood $U \subset \Sigma$ of $v(\C\setminus\{0\}) \cup Q  \cup u(\C),$ we have $u_n(\C) \subset U$ for $n$ sufficiently large. See  \cite{small} and also \cite[Proposition 9.5]{dPS1}. Since $P_2$ is a unique index $2$ orbit, see Proposition \ref{prop:lowwen}, we have $Q=P_2.$
The uniqueness property of finite energy planes asymptotic to $P_2$ and finite energy cylinders asymptotic to $P_3$ at $+\infty$ and to $P_2$ at $0$  (see \cite[Proposition C.-3]{dPS1}),   based on Siefring's intersection theory \cite{Sie11}, implies that   $\tilde v = \tilde v_1$ or $\tilde v=\tilde v_2$ and $\tilde u = \tilde u_{2,1}$ or $\tilde u = \tilde u_{2,2},$ up to reparametrisation and $\R$-translation. 

We now choose a sequence $\tau_n' \to 1^-$ as $n \to +\infty$ and do a similar business with $\tilde u_n' := \tilde u_{\tau_n'}.$ This sequence also converges to a holomorphic building with height 2 whose top level consists of an embedded finite energy cylinder $\tilde v' = (b',v') \colon \C \setminus \{ 0 \} \to \R \times \Sigma$ asymptotic to $P_3$ at its positive puncture $+\infty$ and to $P_2$ at its negative puncture $0$ and whose bottom consists of an embedded finite energy plane $\tilde u' = (a',u') \colon \C \to \R \times \Sigma$ asymptotic to $P_2.$ By uniqueness, we have $\tilde v' = \tilde v_1$ or $\tilde v'=\tilde v_2$ and $\tilde u '= \tilde u_{2,1}$ or $\tilde u '= \tilde u_{2,2},$ up to reparametrisation and $\R$-translation.  Moreover, 
             given a neighbourhood $U' \subset \Sigma$ of $v'(\C\setminus\{0\}) \cup P_2  \cup u'(\C),$ we have $u_n'(\C) \subset U'$ for $n$ sufficiently large.

       \begin{lemma}   $v(\C \setminus \{ 0 \} )=  v' ( \C \setminus \{ 0 \} ) = v_1(\C \setminus \{ 0 \})  ,$   $ u(\C) = u_{2,1}(\C)$ and $ u'(\C) = u_{2,2}(\C). $    
       \end{lemma}
       \begin{proof}

Arguing indirectly we assume that $v(\C\setminus \{ 0 \} ) = v_2 ( \C \setminus \{ 0 \}).$
     Recall that it is obtained as the $\Sigma$-part of a limit of the sequence $\tilde u_{\tau_n} = ( a_{\tau_n}, u_{\tau_n})$ with $\tau_n \to 0^+.$ 
     Given a neighbourhood $U$ of $v_2(\C \setminus \{ 0 \}) \cup P_2 \cup u (\C)$, 
     we have      $u_{\tau_n}(\C) \subset U$ for all $n$ sufficiently large. 
    Since the $x_2$-components of $\tilde u_{3,1}$, $\tilde v_2$ and $\tilde u$ satisfy $x_2 \equiv 0$, $2\pi < x_2 < 4\pi$ and $x_2 \equiv 2\pi,$ respectively, this implies that there is $\tau_0 \in (0,1)$ such that $u_{\tau_0}$ admits a point with  $x_2 = 2\pi.$ Thus, $u_{\tau_0}$ intersects $u_{2,1}$ or $u_{2,2},$ contradicting the  stability and positivity of intersections of immersed pseudoholomorphic curves (see \cite[Appendix E]{MSbook}). We conclude that $v(\C \setminus \{ 0 \} ) = v_1(\C \setminus \{ 0 \}).$  By the same reasoning, we obtain   $v'(\C \setminus \{ 0 \}) = v_1(\C \setminus \{ 0 \}).$

       We now show that the planes $u$ and $u'$ are disjoint.   We follow  the argument given in the proof of \cite[Proposition 3.17]{Oli21}. Assume by contradiction that   they have a non-empty intersection. Then Theorem 2.4 in \cite{Sie11} shows that $u(\C) = u'(\C).$       Pick any point $q \in \Sigma \setminus \left( P_3 \cup v(\C \setminus \{ 0 \}) \cup P_2 \cup u(\C) \right)$ and take a small neighbourhood  $\mathcal{U}$ of $P_3 \cup v(\C \setminus \{ 0 \} ) \cup P_2 \cup u(\C),$ not containing   $q.$ Then there exists $\tau_0$ and $ \tau_1$ contained in $(0,1),$ close enough to $0$ and $1,$ respectively, such that   $u_{\tau_0} (\C) , u_{\tau_1} (\C) \subset \mathcal{U}.$     By Jordan-Brouwer separation theorem, the piecewise smooth embedded two-sphere $S = u_{\tau_0}(\C) \cup P_3 \cup u_{\tau_1}(\C)$ divides $\Sigma$ into two components $C_1$ and $C_2$ having boundary $S.$ It follows from the choice of the point $q$ that one of the two components, say $C_1,$  contains $q$ in its interior and the other contains $P_3 \cup v(\C \setminus \{ 0 \}) \cup P_2 \cup u(\C) .$

    We claim that the  set   $A= \left( \bigcup_{\tau \in (0,1)}   u_{\tau}(\C)  \right) \cap {\rm int}(C_1)$ is non-empty, open and closed in ${\rm int}(C_1).$ In view of Theorem \ref{HWZThm} it is non-empty and open. 
      In order to achieve the closedness of $A$,   take a sequence $w_n \in A.$   Since $w_n$ is contained the closed subset $C_1,$   a limit of the sequence $w_n$ is contained in $C_1.$ On the other hand, since $w_n$ is also contained in the union $\bigcup_{\tau \in (0,1)}   u_{\tau}(\C),$  
      the compactness property of the family $\{u_{\tau}\}_{\tau \in (0,1)}$ described above implies  that  its limit is contained    either in  $\bigcup_{\tau \in (0,1)}   u_{\tau}(\C),$   in $P_3,$ or   in $v(\C \setminus \{ 0 \} ) \cup P_2 \cup u(\C).$ The last option does not occur because  $v(\C \setminus \{ 0 \} ) \cup P_2 \cup u(\C)$ is contained in the interior of $C_2.$ 
     Therefore,    a limit of the sequence $w_n$ is contained in   $\left(   \bigcup_{\tau \in (0,1)}   u_{\tau}(\C)    \cup P_3 \right) \cap C_1,$  from which 
  we see that the set  $\left(   \bigcup_{\tau \in (0,1)}   u_{\tau}(\C)    \cup P_3 \right) \cap C_1$ is  closed in $\Sigma.$ This proves the claim.

      The claim above implies that $q \in \bigcup_{\tau \in (0,1)}   u_{\tau}(\C) .$  
It follows from the choice of  $q$ that
\[
\Sigma =   P_3 \cup v(\C \setminus \{ 0 \}) \cup P_2 \cup u(\C)   \cup  \bigcup_{\tau \in (0,1)}   u_{\tau}(\C)  ,
\]
and hence there exists $\tau_0 \in (0,1)$ such that $u_{\tau_0}$ admits a point with $x_2 = 2\pi,$ which derives a contradiction as above.  Consequently, we find  $u(\C)  \cap u'(\C) = \emptyset. $   Since  $\tau$ increases  in the direction of the Reeb vector field, we have $u(\C) = u_{2,1}(\C)$ and $u'(\C) = u_{2,2}(\C).$ This completes the proof.   
      \end{proof}

Abbreviate
\[
  \mathfrak{B}=  \{ (x_1, x_2, y_1, y_2) \in \Sigma \mid x_2 \leq 2\pi\},
  \]
  whose boundary is the piecewise smooth embedded two-sphere $u_{2,1}(\C) \cup P_2 \cup u_{2,2}(\C).$
 The proof of the previous lemma shows that  the images of  the planes $u_{2,1}, u_{2,2}$ and $u_{\tau}, \tau \in (0,1)$ and the cylinder  $v_1$    determine a smooth foliation of $\mathfrak{B} \setminus ( P_2 \cup P_3).$ 
Since the finite energy plane $\tilde u_{3,2}$ corresponds to $x_2 \equiv0,$ we have $u_{3,2}(\C)  \subset  \mathfrak{B}\setminus ( P_2 \cup P_3).$   Therefore, there is $\tau_* \in (0,1)$ such that $u_{\tau_*}(\C) \cap u_{3,2}(\C) \neq \emptyset.$ 
We find again by  \cite[Theorem 1.4]{HWZII} that $u_{\tau_*}(\C) = u_{3,2}(\C).$

Note that
\[
\rho (\mathfrak{B} ) = \overline{ \Sigma \setminus \mathfrak{B}} 
= \{ (x_1, x_2, y_1, y_2) \in \Sigma \mid x_2 \geq 2\pi\}.
\]
We then push all the data above to the region $\rho (\mathfrak{B} ).$ More precisely, we have a one-parameter family of embedded finite energy $\tilde J$-holomorphic planes
\[
(\tilde u_{\tau})_{\rho} \colon \C \to \R \times \Sigma, \quad \tau \in (\tau_-, \tau_+),
\]
all asymptotic to $P_3 ' = \rho(P_3).$ 
Note that $(u_{\tau_*})_\rho(\C) = u_{3,2}'(\C).$
This family breaks onto $\tilde v_2 =(\tilde v_1)_\rho$ and $\tilde u_{2,1}  = (\tilde u_{2,1})_\rho$ at one end and onto $\tilde v_2 $ and $\tilde u_{2,2}  = (\tilde u_{2,2})_\rho$ at the other end. The images of maps $u_{2,1},u_{2,2}, (u_{\tau})_\rho, \tau \in (\tau_-, \tau_+)$ and $v_2$ fill the three-ball $\rho(\mathfrak{B}).$

       We have constructed a stable finite energy   foliation $\tilde{\mathcal{F}}$  of $\R \times \Sigma,$  which is symmetric with respect to the exact anti-symplectic involution $\tilde \rho$ and whose leaves are   embedded finite energy planes and cylinders.  
    Denote by $\mathcal{F}$ the projection  of $\tilde{\mathcal{F}}$ to $\Sigma,$ which is symmetric with respect to the anti-contact involution $\rho.$ 
    Since finite energy planes in $\tilde{\mathcal{F}}$ are asymptotic to   $P_2$ or $P_3,$ their $\Sigma$-projections are transverse to the Reeb vector field \cite{Hry12fast}. See also \cite[Section 13.6]{FvK18book}. It is straightforward to see that  the $\Sigma$-parts $v_1, v_2$ of the cylinders  $\tilde v_1 , \tilde v_2$ are transverse to the Reeb vector field as well. Therefore,
  $\mathcal{F}$   
    determines a transverse foliation of $\Sigma$ whose    binding orbits are $P_2 ,P_3$ and $P_3'$ and whose regular leaves are embedded planes and cylinders.   See Figure \ref{fig:folaitionspre}.

 \begin{figure}[ht!]
  \centering
  \includegraphics[width=0.7\linewidth]{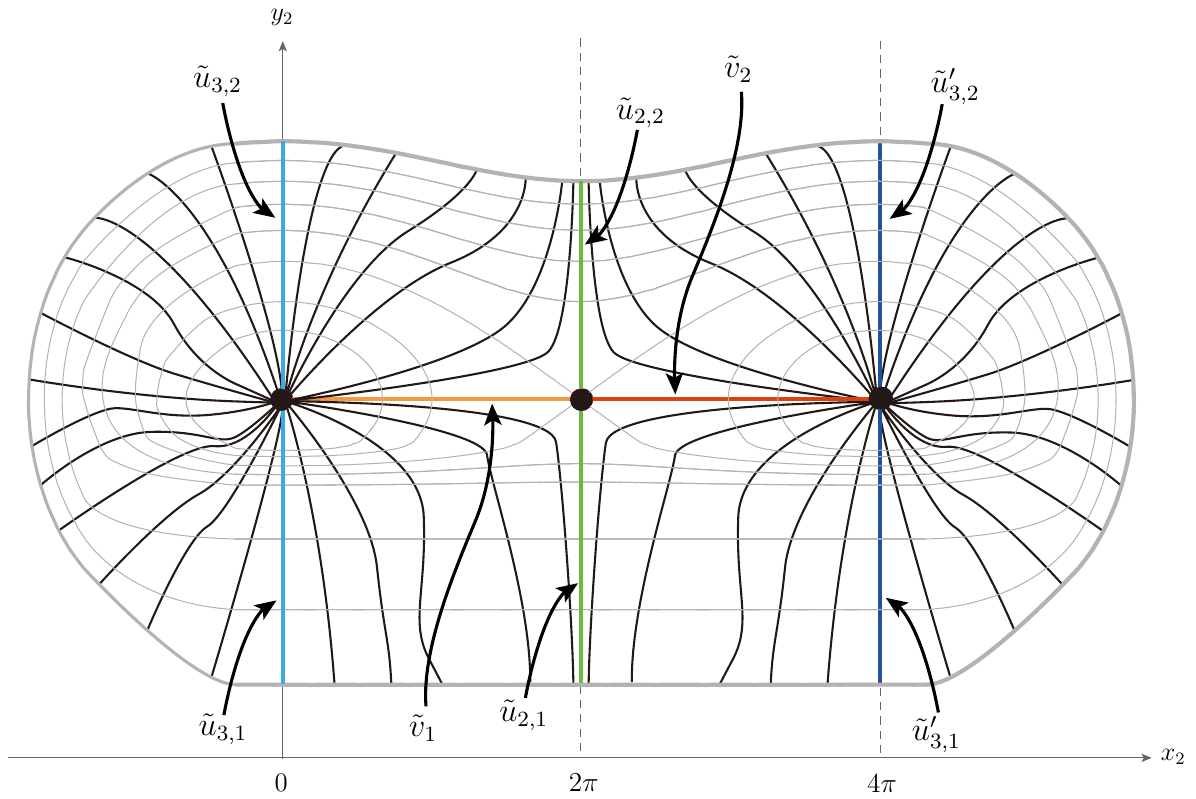}
    \caption{The projection of the transverse foliation $\mathcal{F}$ to the $(x_2,y_2)$-plane. The black curves indicate the families of planes.     }
    \label{fig:folaitionspre}
\end{figure}

\begin{remark}
    A transverse foliation satisfying the properties above is called a weakly convex foliation. For more details, see \cite{Weakconvex}.
\end{remark}

We now cut the energy level $\Sigma$ along the two-spheres 
\begin{align*}
B &= \{ (x_1, x_2, y_1, y_2) \in \Sigma \mid x_2 = 0 \}  = u_{3,1}(\C) \cup P_3 \cup u_{3,2}(\C),\\  
\rho(B) &= \{ (x_1, x_2, y_1, y_2) \in \Sigma \mid x_2 = 4 \pi \}  = u_{3,1}'(\C) \cup P_3' \cup u_{3,2}'(\C). 
\end{align*}
Denote by $\Sigma_0$ a unique component whose boundary is given by $B \cup \rho(B).$ 
It is diffeomorphic to $S^2 \times I,$ where $I$ is a closed interval. 
We then identify   $B$ and $\rho(B) ,$ providing    us with  $\bar{\Sigma} \cong S^2 \times S^1.$ Recall that $\bar{\Sigma}$ is a unique connected component of $\bar{H}^{-1}(c)$ which contains $\bar{\Sigma}_{\rm direct}^u,$ the double cover of $\Sigma_{\rm direct}^u.$ See Section \ref{sec:interpol}.  
 The contact form $\lambda,$ constructed in Section \ref{sec:form}, induces a contact form $\bar{\lambda}$ on $\bar{\Sigma},$ and the anti-contact involution $\rho$ induces an anti-contact involution on the contact manifold $(\bar{\Sigma}, \bar{\lambda}).$ We   denote it again by $\rho.$ The unique index-2 periodic orbit $P_2$ still lies on $\bar{\Sigma},$ and the two index-3 binding orbits $P_3$ and $P_3'$ now provide a single symmetric periodic index-3 orbit, still denoted by $P_3.$ The periodic orbit $P_0$ yields two periodic orbits, one satisfies $y_2 = \Lambda_{\max}$ and the other satisfies $y_2 \equiv \Lambda_1.$ Note that the latter is the second iterate of the direct circular orbit $\gamma_{\rm direct}^u.$

The almost complex structure $J \colon \xi \to \xi$ as in \eqref{eq:Jcomp} induces a $d\bar{\lambda}$-compatible and $\rho$-anti-invariant  almost complex structure $\bar J \colon \bar{\xi} \to \bar{\xi},$ where $\bar{\xi} = \ker \bar{\lambda}. $ Hence, the finite energy foliation $\tilde{\mathcal{F}}$ of $\R \times \Sigma$ induces a finite energy foliation of $\R \times \bar{\Sigma}$ consisting of the following:
\begin{itemize}

\item  finite energy planes $\tilde u_{3,1}$ and $\tilde u_{3,2},$   corresponding to $x_2 \equiv 0  $ (mod $4\pi$) and initial conditions $y_2(0) \in (\Lambda_1, \Lambda_3)$ and $y_2(0) \in (\Lambda_3, \Lambda_{\max}),$ respectively.
They satisfy $\tilde u_{3,j} = ( \tilde u_{3,j})_\rho, j=1,2$ and are both asymptotic to $P_3.$

\item  finite energy planes $\tilde u_{2,1}$ and $\tilde u_{2,2},$ corresponding to $x_2 \equiv 2 \pi$  (mod $4\pi$) and initial conditions $y_2(0) \in (\Lambda_1, \Lambda_3)$ and  $y_2(0) \in (\Lambda_3, \Lambda_{\max}),$ respectively. 
They satisfy $\tilde u_{2,j} = ( \tilde u_{2,j})_\rho, j=1,2$ and are both asymptotic to $P_2.$

\item  finite energy cylinders $\tilde v_1$ and $ \tilde v_2 = (\tilde v_1)_\rho,$ which corresponds to $y_2 \equiv \Lambda_3 $ and initial conditions $x_2(0) \in (0, 2\pi)$ and $x_2(0) \in (2 \pi, 4 \pi),$ respectively.  
They are  both asymptotic to $P_3$ at  $+\infty$ and to $P_2$ at   $-\infty.$

\item  two families of finite energy planes  $(\tilde{u}_{j, \tau}),$ $\tau \in (0,1).$ 
For each $j=1,2,$ the family $(\tilde{u}_{j, \tau})$ breaks onto $\tilde v_1$ and $\tilde u_{2,j}$ at one end and onto $\tilde v_2$ and $\tilde u_{2,j}$ at the other end.

\end{itemize}
The projection of the finite energy foliation above defines a transverse foliation $\bar{\mathcal{F}}$ of $\bar{\Sigma},$ whose binding orbits are $P_2 $ and $P_3$ and whose regular leaves are embedded planes and cylinders.

 Denote by $\mathcal{T} $ the compact subset of the energy level $\bar \Sigma$, which projects into the $(x_2, y_2)$-plane (modulo $4\pi$ in $x_2$) as the set
\[
\{  (x_2, y_2) \mid   x_2  \in \R / 4 \pi \Z,\ \Lambda_1 \leq y_2 \leq \Lambda_3\}.
\]
        Note that  
        the double cover   $\bar{\Sigma}_{\rm direct}^u$ of    the set $\Sigma_{\rm direct}^u$ under consideration  is contained in $\mathcal{T}  $ and that $\mathcal{T} $ is diffeomorphic to a solid torus and has boundary 
       \[
       \partial \mathcal{T}= P_3 \cup P_2    \cup v_1(\C \setminus \{ 0 \}) \cup v_2 ( \C \setminus \{ 0 \})
       \]
       lying over the set $    \{ (x_2, \Lambda_3) \mid x_2 \in \R / 4 \pi \Z \}.$ 
        The interior of $\mathcal{T}$ is filled with the images of planes $u_{2,1}$ and  $u_{\tau}=u_{1,\tau}, \tau \in (0,1).$  The core of this solid torus is the second iterate of the  direct circular orbit $\gamma^u_{\rm direct}.$

       Let $F \in \bar{\mathcal{F}}$ be a plane asymptotic to $P_3$ that is not rigid and    contained in $\mathcal{T}_1.$     
       For $j=1,2,$ denote 
         $\Pi_j (x_1, x_2, y_1, y_2) = (x_j, y_j), (x_1, x_2, y_1, y_2) \in \bar \Sigma.$
       Then $\Pi_1(F) $ is  a closed disc of radius $r_3$ centred at the origin and $\Pi_2(F)$ is an arc connecting $(x_2, \Lambda_1)$ for some $  x_2 \in (0,4\pi) $ and $(0, \Lambda_3) \equiv ( 4\pi, \Lambda_3).$      If $(x_1, x_2, y_1, y_2) \in \bar{\Sigma}$ is such that $(x_2, y_2) \in \Pi_2(F)$ with $y_2 <\Lambda_2,$   then $(x_1, y_1)$ lies along the circle
       \begin{equation}\label{eq:x1x2y1y2}
       x_1^2 + y_1^2 = 2\left(c + \frac{1}{2y_2^2}+y_2\right).
       \end{equation} See \eqref{eq:orihamlow}.
       Note that the right-hand side above equals zero if and only if $y_2 =\Lambda_1.$

 In order to show that $\bar{\mathcal{F}}$ intersects  the set $\bar{\Sigma}_{\rm direct}^u$ in embedded discs, we have to rule out the presence of a regular leaf  $F \in \bar{\mathcal{F}}$ having the property that there is $y_2 <\Lambda_2$ such that $\# [\Pi_2(F)\cap \{ (x_2, y_2) \mid 0<x_2<4\pi\}]>1,$ as
 illustrated in Figure \ref{fig:no}.   
   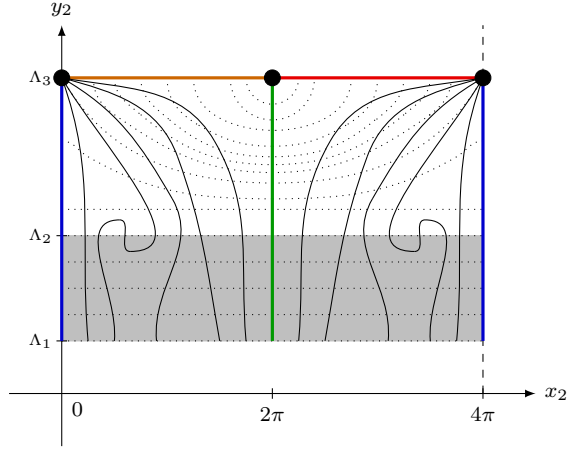
\begin{figure}[ht!]
     \centering
\begin{tikzpicture} [scale=0.7]

 \filldraw[draw=lightgray,fill=lightgray] (0,1) rectangle (8,3);

 \draw[dashed] (8,0) to (8,7);
\draw[->] (-1,0) to (9 ,0);
\draw[->] (0,-1) to (0,7);
 
\begin{scope}[yshift=-1cm]
  
\begin{scope}[yscale=-1, yshift=-14cm]
\draw[  dotted]  (0,7) to [out=60, in=180] (4.057,9);
\begin{scope}[xscale=-1, xshift=-8cm]
\draw[ dotted]  (0,7) to [out=60, in=180] (4,9);
\end{scope}
\end{scope}

\begin{scope}[yscale=-1, yshift=-14cm]
 \draw[ dotted]    (3.5,7) arc (180:0:0.5cm and 0.5cm);
\draw [ dotted]   (3.2,7) arc (180:0:0.8cm and 0.8cm);
 \end{scope}
 
\begin{scope}[yscale=-1, yshift=-14cm]
\draw[, dotted] (2.5, 7) to [out =90, in=180]  (4.057, 8);

\draw[ dotted] (2 , 7) to  [out =90, in=210](2.6,7.9);
\draw[ dotted] (2.6,7.9) to  [out =30, in=180](4, 8.3);

\draw[ dotted] (1.5, 7) to  [out =90, in=220](2.05, 7.8 );
\draw[ dotted] (2.05, 7.8) to  [out =40, in=180](4, 8.5);

\draw[ dotted] ( 1, 7) to [out =90, in=220] (1.5, 7.7 );
\draw [ dotted]( 1.5, 7.7) to [out =40, in=180] (4, 8.65);

\draw[ dotted] (0.6, 7) to  [out =90, in=215](1, 7.7);
\draw[ dotted] (1, 7.7) to  [out =35, in=180](4, 8.8);

\begin{scope}[xscale=-1, xshift=-8cm]
\draw[ dotted] (2.5, 7) to [out =90, in=180]  (4, 8);

\draw[ dotted] (2 , 7) to  [out =90, in=210](2.6,7.9);
\draw[ dotted] (2.6,7.9) to  [out =30, in=180](4, 8.3);

\draw[ dotted] (1.5, 7) to  [out =90, in=220](2.0584, 7.82 );
\draw[ dotted] (2.05, 7.8) to  [out =40, in=180](4, 8.5);

\draw[ dotted] ( 1, 7) to [out =90, in=220] (1.5, 7.7 );
\draw[ dotted] ( 1.5, 7.7) to [out =40, in=180] (4, 8.65);

\draw[ dotted ] (0.6, 7) to  [out =90, in=215](1.05, 7.73);
\draw[ dotted] (1, 7.7) to  [out =35, in=180](4, 8.8);
\end{scope}
\end{scope}
   
\begin{scope}[yscale=-1, yshift=-14cm]
\draw[ dotted,]     (0,8.2 ) to  [out=0 , in=210]  (0.5, 8.5);
\draw[, dotted  ]     (0.5,8.5) to  [out=30 , in=180]  (4.04, 9.3);
\begin{scope}[xscale=-1, xshift=-8cm]
\draw[, dotted]     (0,8.2 ) to  [out=0 , in=210]  (0.5, 8.5);
\draw[ dotted  ]     (0.5,8.5) to  [out=30 , in=180]  (4.04, 9.3);
\end{scope}
\end{scope}
\end{scope}

\draw[  dotted ]     (0,1) to    (8, 1);
\draw[ dotted ]     (0,1.5) to    (8, 1.5);
\draw[ dotted ]     (0,2) to    (8, 2);
\draw[ dotted ]     (0,2.5) to    (8, 2.5);
\draw[ dotted ]     (0,3) to    (8,3);
\draw[ dotted ]     (0,3.5) to    (8, 3.5);
 \draw (4,-0.1) to (4,0.1);
 \draw (8,-0.1) to (8,0.1);
 \draw (-0.1, 1) to (0.1,1);
 \draw (-0.1, 3) to (0.1,3);
 
  \draw[very thick, black!20!blue] (0,1) to (0,6);
  \draw[very thick, black!20!blue] (8,1) to (8,6);
  
    \draw[very thick, black!20!orange] (0,6) to (4,6);
    \draw[very thick, black!10!red] (8,6) to (4,6);

      \draw[very thick, black!40!green] (4,1) to (4,6);

\begin{scope}[xscale=-1, xshift=-12cm]
\begin{scope}[xscale=-1, xshift=-8cm]
\draw  (0.5,1) [out=85, in=240] to (1.1,4.8 );
\draw (1.1, 4.8 ) [out=60, in=190] to (4,6);
\draw  (1,1) to [out=80, in=250]   (1.8,4.5);
\draw  (1.8,4.5) [out=70, in=195] to (4,6);

\end{scope}

\draw  (4.5,1) [out=95, in=290] to (4,6);

\draw  (5,1) to [out=85, in=180] (5.1, 3.3);
\draw (5.1, 3.3) to [out=0, in=180] (5.3, 2.7);
\draw (5.3, 2.7) to [out=0, in=300] (4,6);
 
 \draw  (5.8,1) to [out=95, in=300](6.1,3.7);
\draw (6.1, 3.7) to [out=120, in=330] (4,6);

\end{scope}

\begin{scope}[xscale=-1, xshift=-8cm]

\begin{scope}[xscale=-1, xshift=-12cm]
\begin{scope}[xscale=-1, xshift=-8cm]
\draw  (0.5,1) [out=85, in=240] to (1.1,4.8 );
\draw (1.1, 4.8 ) [out=60, in=190] to (4,6);
\draw  (1,1) to [out=80, in=250]   (1.8,4.5);
\draw  (1.8,4.5) [out=70, in=195] to (4,6);

\end{scope}

\draw  (4.5,1) [out=95, in=290] to (4,6);

\draw  (5,1) to [out=85, in=180] (5.1, 3.3);
\draw (5.1, 3.3) to [out=0, in=180] (5.3, 2.7);
\draw (5.3, 2.7) to [out=0, in=300] (4,6);
 
 \draw  (5.8,1) to [out=95, in=300](6.1,3.7);
\draw (6.1, 3.7) to [out=120, in=330] (4,6);

\end{scope}
\end{scope}




 \node at (0,6) [left] {\tiny $\Lambda_3$};
  \node at (0,1) [left] {\tiny$\Lambda_1$};
 \node at (0,3) [left] {\tiny$\Lambda_2$};
 \node at (0.3,0) [below] {\footnotesize$0$};
\node at (4,-0.08) [below] {\footnotesize$2\pi$};
 \node at (8,-0.08) [below] {\footnotesize$ 4\pi$};
\node at (0,7) [above] {\footnotesize$y_2$};
\node at (9,0)[right] {\footnotesize$x_2$};
 \draw[fill ] (0,6) circle (0.15cm);
\draw[fill ] (8 ,6) circle (0.15cm);
\draw[fill ] (4,6) circle (0.15cm);
  \end{tikzpicture}
    \caption{An unpleasant scenario that is not the case.}
    \label{fig:no}
 \end{figure}  
To this end, fix a regular leaf $F \in \bar{\mathcal{F}},$ asymptotic to $P_3$ and contained in $\mathcal{T}  \setminus \partial \mathcal{T} .$ Let $\tilde u = (a,u)  $ be a corresponding finite energy plane. Denote  $u(s,t) = (x_1(s,t), x_2(s,t), y_1(s,t), y_2(s,t))$ as before. Assume by contradiction that there is $(s_0, t_0)  $ such that $ \Lambda_1 < y_2(s_0, t_0) \leq \Lambda_2$ and $  (y_2)_s(s_0, t_0) = ( y_2)_t(s_0, t_0)=0.$ Introduce polar coordinates $(r, \theta)$ in the $(x_1, y_1)$-plane. 
By means of  \eqref{eq:x1x2y1y2} we find that
\[
r(s,t) r_\sigma(s, t) = -  (y_2)_\sigma (s,t) \left( \frac{1}{y_2(s,t)^3} -1\right), \quad \sigma = s, t.
\]
Since $y_2(s_0, t_0)> \Lambda_1 >1,$ we find $r(s_0,t_0)>0,$ and hence  $  r_s(s_0,t_0) =  r_t(s_0,t_0)=0.$ Let $\gamma(t) = (x_1(t), x_2(t), y_1(t), y_2(t))$ be a Hamiltonian trajectory that intersects $F$ at $u(s_0,t_0).$ Since $y_2 \leq \Lambda_2,$ without loss of generality we may assume that $\bar{H} = H, $  where the latter is as in \eqref{eq:orihamlow}. Since $H$ is free of the variable $x_2,$ we find using Hamilton's equations that $y_2(t) = y_2(s_0, t_0), \forall t.$ Moreover, $\Pi_1(\gamma(t))$ moves along the circle of radius $r(s_0,t_0).$  The discussion so far shows that $F$ and $\gamma$ have a non-empty tangent intersection at $u(s_0, t_0),$ contradicting the fact that $F$ is transverse to the Hamiltonian flow.

       It remains to show that the associated first return map admits a unique fixed point corresponding to the direct circular orbit $\gamma_{\rm direct}^u.$ It is obvious that $\gamma_{\rm direct}^u$ corresponds to a fixed point.
       Recall that every periodic orbit $\gamma \neq \gamma_{\rm direct}^u$   in $\Sigma_{\rm direct}^u$   is a $T_{k,l}$-type orbit for some coprime positive integers $k,l.$ See Section \ref{sec:RKP}. It is easy to see that $k<l$ on $\Sigma_{\rm direct}^u.$ See for instance \cite[Section 6]{RKP13} or \cite[Section 3]{KimRKP}. This implies that the periods of $\Pi_1\gamma$ and $\Pi_2 \gamma$ do not coincide, and hence every $T_{k,l}$-type orbit  intersects each leaf at least twice.

                We conclude that the transverse foliation $\bar{  \mathcal{F}}$  of $\bar \Sigma$ intersects  $\Sigma_{\rm direct}^u$ in embedded discs, providing a disc foliation of $\Sigma_{\rm direct}^u.$ Moreover, the associated first return map has a unique fixed point corresponding to $\gamma_{\rm direct}^u.$ This finishes the proof of Theorem \ref{thm:main} for energies below the critical value.

       \begin{remark}\label{rmk:app3}
       In Appendix \ref{sec:app} we will embed $\Sigma_{\rm direct}^u$ into a closed three-manifold $\hat \Sigma \cong S^2 \times S^1$ and then construct its finite energy foliation, projecting to an open book decomposition of $\hat \Sigma$ with annulus-like pages. The intersection of  this open book decomposition with  $\Sigma_{\rm direct}^u$ provides   a  foliation of   $\Sigma_{\rm direct}^u$  
       into annuli, whose inner boundaries are the direct circular orbit $\gamma_{\rm direct}^u$ and the outer boundaries lie on the boundary torus $\p \Sigma_{\rm direct}^u.$     The construction is explicit and does not require the implicit function theorem and compactness argument. 
       \end{remark}

      \section{Above the critical energy}\label{sec:higher}
  
  
  Fix $-\frac{3}{2} <c<0$ and let   $E_0    \in (E_{\rm retro}^u,c).$  Recall that  $E =c$ corresponds to the collision orbits. This section is devoted to prove the assertion of Theorem \ref{thm:main} for the set
\begin{equation*}\label{defset2}
     \Sigma _{\rm retro}^u = \{\text{all periodic orbits of $X_H$ with   $E \in [E_{\rm retro}^u, E_0]$}\} \subset \Sigma_c^u,
     \end{equation*}    
     which is diffeomorphic to a solid torus with core being the retrograde circular orbit $\gamma_{\rm retro}^u.$ Note that it  contains   only retrograde orbits and no collision orbits.  

  In this case   the transformed  Hamiltonian is given by
      \begin{equation*}\label{eq:orihamd}
      H (x_1, x_2 , y_1, y_2) = - \frac{1}{2y_2^2} - y_2 - \frac{1}{2} ( x_1 ^2 + y_1 ^2)  
      \end{equation*}
      for $(x_1, x_2, y_1, y_2) \in   \overline{\mathcal{P}}' \subset \R \times \R/ 2 \pi \Z \times \R \times \R_-.$  
      For convenience, we shall work with $-H$ instead of $H,$ so that every orbit is oriented in the  opposite direction. Note that the qualitative behaviour of the dynamics does not change. By abuse of notation, we denote $-H$ again by $H,$ namely, we have
        \begin{align*}\label{eq:orihamd2}
     \begin{split}
      H (x_1, x_2 , y_1, y_2) &=  \frac{1}{2y_2^2} + y_2 + \frac{1}{2} ( x_1 ^2 + y_1 ^2)  \\
      &= H_1(x_1, y_2) - H_2(y_2),
      \end{split}
     \end{align*}
      where $H_1$ and $H_2$ are as in \eqref{eq:H1H2}. 
        Note that  
      \[
      y_2 \in [\nu_2, \nu_1] \quad \text{ on } \  \Sigma _{\rm retro}^u, 
      \]
      where    $\nu_1 =  - \sqrt{ - \frac{1}{ 2E_{\rm retro}^u}}  \in  ( - \frac{1}{\sqrt[3]{2}}, -\frac{1}{2})$ and $\nu_2 = - \sqrt{ - \frac{1}{2E_0}}.$

 Arguing as in Section \ref{sec:discfol}   one can construct a disc foliation of $\Sigma_{\rm retro}^u$   by looking at the negative gradient flow lines of $-H_2.$    See Figure \ref{fig:ngflK3}. 
   \begin{figure}[ht]
     \centering
\begin{tikzpicture}

 \draw[->] (-0.5,3) to (5,3);
 \draw[->] (0,0) to (0,3.5);

 \filldraw[draw=lightgray,fill=lightgray] (0,0.5) rectangle (4,2.5);
 \draw[dashed] (2,3) to (2,2.5);
 \draw[dashed] (4,3) to (4,2.5);
 \node at ( -0.2,3) [above] {$0$};
 \node at (2,3) [above] {$ \pi$};
 \node at (4,3) [above] {$2\pi$};
 \node at (0,0.5) [left] {$\nu_2$};
 \node at (0,2.5) [left] {$\nu_1$};
 \node at (0,3.5) [above] {$y_2$};
 \node at (5,3) [right] {$x_2$};
 
 \draw[dashed, gray] (0,0.5) to (4, 0.5);
 \draw[dashed, gray] (0,0.9) to (4, 0.9);
 \draw[dashed, gray] (0,1.3) to (4, 1.3);
 \draw[dashed, gray] (0,1.7) to (4, 1.7);
 \draw[dashed, gray] (0,2.5) to (4, 2.5);
 \draw[dashed, gray] (0,2.1) to (4,2.1);
  
 \draw[blue, thick, ->] (0.5,2.5) to (0.5, 1.5);
 \draw[blue, thick] (0.5,2.5) to (0.5, 0.5);
 
  \draw[blue,thick, ->] (0 ,2.5) to (0 , 1.5);
 \draw[blue, thick] (0 ,2.5) to (0 , 0.5);
 
  \draw[blue, thick, ->] (1,2.5) to (1, 1.5);
 \draw[blue, thick] (1,2.5) to (1, 0.5);
 
  \draw[blue, thick, ->] (1.5,2.5) to (1.5, 1.5);
 \draw[blue, thick] (1.5,2.5) to (1.5, 0.5);
 
  \draw[blue, thick, ->] (2,2.5) to (2, 1.5);
 \draw[blue, thick] (2,2.5) to (2, 0.5);
 
  \draw[blue, thick, ->] (2.5,2.5) to (2.5, 1.5);
 \draw[blue, thick] (2.5,2.5) to (2.5, 0.5);
 
   \draw[blue,thick, ->] (3.5,2.5) to (3.5, 1.5);
 \draw[blue,thick] (3.5,2.5) to (3.5, 0.5);
 
   \draw[blue,thick, ->] (3,2.5) to (3, 1.5);
 \draw[blue,thick] (3,2.5) to (3, 0.5);
 
   \draw[blue,thick, ->] (4,2.5) to (4, 1.5);
 \draw[blue,thick] (4,2.5) to (4, 0.5);

  \end{tikzpicture}
    \caption{Negative gradient flow lines of $-H_2.$ Dashed lines indicate Hamiltonian trajectories.}
    \label{fig:ngflK3}
 \end{figure}
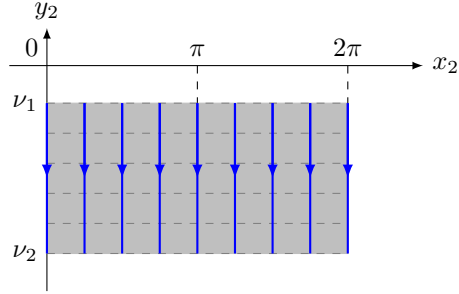  
 Let $h_{x_2}$ be the negative gradient flow line of $-H_2$ corresponding to $x_2 \in \R / 2\pi \Z.$  Its preimage $\Pi^{-1}(h_{x_2})$ under the projection $\Pi(x_1, x_2, y_1, y_2) = (x_2, y_2), (x_1, x_2, y_1, y_2) \in \Sigma_{\rm retro}^u$ is an embedded disc, transverse to the Hamiltonian flow. By varying $x_2$ in $\R / 2 \pi \Z,$ we find a disc foliation of $\Sigma_{\rm retro}^u.$ 
 

     Fix $\nu_3 \in (-\infty, \nu_2)$  and $0< \varepsilon_0 < \frac{1}{2} (\nu_2 - \nu_3). $ Consider
      \[
      U(x_2, y_2) =  \frac{1}{2} ( y_2 - \nu_3)^2  - \cos \frac{x_2}{2} +D,
      \]
      where
     $(x_2, y_2) \in \R / 4 \pi \Z \times \R_+,$ and  
      $D$ is a constant satisfying
      \begin{equation*}\label{eq:sec6B}
     D<   \min\{0,  \frac{1}{2 ( \nu_2-\varepsilon_0)^2}  + (\nu_2 - \varepsilon_0)  - \frac{1}{2} ( \nu_2 -  \varepsilon_0 - \nu_3)^2 -1 \} .
      \end{equation*}
      Set 
      \begin{equation*}
          V(x_2, y_2) =- \left( 1-f(y_2) \right) H_2(y_2) + f(y_2) U(x_2, y_2),
      \end{equation*}
where a non-increasing smooth function $f \colon \R \to [0,1]$ satisfies   $f(y_2) = 1, \forall  y_2 < \nu_2 - 2\varepsilon_0$ and $f(y_2)=0, \forall y_2 > \nu_2 - \varepsilon_0.$ 
      See Figure \ref{fig4}.
        \begin{figure}[ht]
  \centering
  \includegraphics[width=0.5\linewidth]{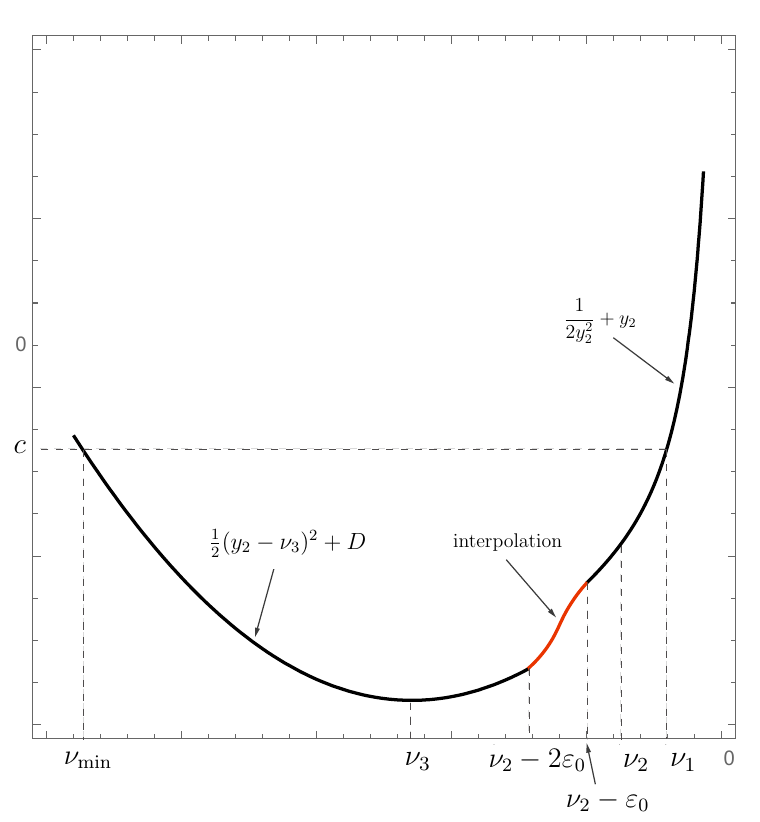}
  \caption{Interpolation for variable $y_2$ }
  \label{fig4}
\end{figure}
Because of the choice of the constant $D,$ we have
\[
V_{y_2} <0, \ \forall y_2 < \nu_3 \quad \text{ and } \quad V_{y_2} >0, \forall y_2 < \nu_3,
\]
and $(0, \nu_3)$ and $(2\pi, \nu_3)$ are the only critical points of $V.$ 
Given $x_2  ,$ denote by $\nu_{\min} = \nu_{\min}(x_2)$ the value of $y_2$ satisfying $V(x_2, \nu_{\min})=-c.$  
 
      
 We define 
 \begin{equation*}
     \bar{H}(x_1, x_2, y_1, y_2) = H_1(x_1, y_1) + V(x_2, y_2) ,
 \end{equation*}
where $(x_2, y_2) \in \R / 4 \pi \Z \times \R_+,$ and note that it is invariant under the anti-symplectic involution $\rho$ as in \eqref{eq:rhofirst}.

We now break the periodicity in $x_2$ precisely as in Section \ref{sec:interpol}. Namely, we  consider
\[
W(x_2, y_2)=\frac{1}{2}x_2^2 + \frac{1}{2} (y_2 -\nu_3)^2
\]
and interpolate
\[
\tilde V (x_2, y_2) = \left( (1-g(x_2) \right) W(x_2, y_2) + g(x_2) V(x_2, y_2),
\]
where $g$ is a non-decreasing smooth function $g\colon \R \to [0,1]$ such that $g(x_2)=0, \forall x_2 < -2\varepsilon_1$ and $g(x_2) =1, \forall x_2 >-\varepsilon_1$ with $\varepsilon_1>0$ small enough. By repeating in a similar manner for $x_2 \geq 2\pi$ in a symmetric way, we obtain a smooth function $\tilde H_2(x_2, y_2).$ Note that it is symmetric in $x_2$ with respect to $x_2 =2\pi.$

We then obtain a modified Hamiltonian
\[
\tilde H (x_1, x_2, y_1, y_2) = H_1(x_1, y_1) + \tilde H_2( x_2, y_2).
\]
It is easy to see that $\tilde H$   admits exactly three critical points given in \eqref{eq:criitodfhildH} and satisfies
\begin{equation*}\label{eq:tildehy2negativehighre}
\tilde H_{y_2} <0, \ \forall y_2 < \nu_3, \quad \tilde H_{y_2}>0, \ \forall y_2 >\nu_3.
\end{equation*}
This implies that there is a connected component $\Sigma \cong S^3$ of the energy level $\tilde H^{-1}(-c).$  The anti-symplectic involution $\rho$ induces an anti-symplectic involution of $\R^3 \times \R_-,$ still denoted by $\rho.$ Since $\tilde H$ is invariant under $\rho$, so is $\Sigma.$

The rest of the  proof is analogous to the previous case,  so we omit it. This finishes the proof.

  \section*{Acknowledgments}
    The author would like to express his gratitude to      Urs Frauenfelder and Pedro Salom\~ao for invaluable comments. He  was supported 
    by the National Research Foundation of Korea(NRF) grant funded by the Korea government(MSIT) (No.\ NRF-2022R1F1A1074066).

      \appendix
      
      \section{Annulus foliations}\label{sec:app}

       As announced in Remark   \ref{rmk:app3} we shall construct  a foliation of $\Sigma_{\rm direct}^u$ (when $c<-\frac{3}{2}$) different from the one given in Section \ref{sec:lower}. An annulus foliation of $\Sigma_{\rm retro}^u$ (when $c \in (-\frac{3}{2},0)$) can be constructed in a similar manner.

     We fix $c<-\frac{3}{2}$ and $E_0 \in (E_{\rm direct}^u,0)$ and define $\Sigma_{\rm direct}^u \subset \Sigma_c^u$ as before.
Consider   $ \Lambda_3 >   \Lambda_2 $ and $0< \varepsilon < \frac{1}{2}( \Lambda_3 -\Lambda_2),$ where $\Lambda_2$ is as in   Section \ref{sec:lower}.
Let 
\[
V(y_2) = \frac{1}{2} (y_2 - \Lambda_3)^2 +B,
\]
where the constant $B$ is given as in \eqref{eq:B}, and
define
\[
\hat H_2 (y_2) = ( 1 - f(y_2)) H_2(y_2) + f(y_2) V(y_2),
\]
where     $f \colon \R \to [0,1]$ is a non-decreasing function satisfying
     \[ 
     \begin{cases} 
     f(y_2) =0 &  y_2 < \Lambda_2+  \varepsilon ,\\
     f(y_2)=1 & y_2 >  \Lambda_2 +  2\varepsilon.
     \end{cases}
     \]
Arguing as in Section \ref{sec:interpol}, we see that
\begin{equation}\label{eq:hatH}
    \hat H_2' <0, \ \forall y_2 < \Lambda_3  \quad \text{ and } \quad  \hat H_2'>0, \ \forall y_2 > \Lambda_3.
\end{equation}

We  introduce  
\[
\hat H (x_1, x_2, y_1, y_2)= H_1(x_1, y_1)+ \hat{H}_2 (y_2).
\]
Note that $\hat H$ is free of the variable $x_2.$ 
       The energy level ${\hat{H}}^{-1}(c)$ contains a  connected component $\hat \Sigma \cong S^1 \times S^2,$
  containing $\Sigma_{\rm direct}^u,$ where we have identified $S^1 \equiv \R/2 \pi \Z.$  
    If $(x_1, x_2, y_1, y_2) \in  \hat \Sigma,$ then
\[
\Lambda_1 \leq y_2 \leq \Lambda_{\max} := \sqrt{2(c-B)}  +\Lambda_3 .
\]

 The Liouville vector field
  \begin{equation*}\label{liouvilleee}
  \hat Y = \frac{1}{2} ( x_1 \p_{x_1} + y_1 \p _{y_1} ) +  (y_2 - \Lambda_3) \p_{y_2}  
 \end{equation*}
 satisfies $d\hat{H}(\hat Y)>0,$ implying that it   is transverse to $  \hat \Sigma.$ Therefore,   we obtain  a contact form $\hat\lambda$  
      \[
      \hat\lambda = \omega_0 ( \hat Y, \cdot)|_{\hat \Sigma} = \frac{1}{2} ( y_1 dx_1 -  x_1 dy _1 ) + ( y_2 - \Lambda_3) dx_2.
      \]

Consider the periodic orbits $Q_1 = (w_1, T_1)$ and $Q_2=(w_2, T_2),$ where
\begin{align*}
  w_1(t)  &= ( 0, ( \frac{1}{ \Lambda_1^3} -1)t, 0, \Lambda_1) ,  \quad T_1 = \frac{ 2 \pi} {1- 1/\Lambda_1^3  }\\
  w_2(t) &= (0, \sqrt{ 2 ( c-B)}t, 0, \Lambda_{\max}), \quad T_2 = \frac{ 2\pi}{\sqrt{ 2 ( c-B)}}.
 \end{align*}
Note that $Q_1$ corresponds to the direct circular orbit $\gamma_{\rm direct}^u.$
Since   
\[
\lambda(X_{\hat{H}})=\begin{cases}   ( \Lambda_1 - \Lambda_3)( \frac{1}{\Lambda_1^3}-1) & \text{ along } Q_1,\\
  2(c-B) & \text{ along } Q_2,
  \end{cases}
  \]   their Reeb periods are $2\pi(\Lambda_3-\Lambda_1)$ and $2 \pi \sqrt{ 2(c-B)},$ respectively.

      Introduce polar coordinates $(r, \theta)$ in the $(x_1, y_1)$-plane. Note that $ r = \sqrt{ 2(c-\hat{H}_2)}, $ showing that  the variable $r$ takes the minimum $r=0$ when $\hat{H}_2=c,$ or equivalently, $y_2 \in \{ \Lambda_1, \Lambda_{\max}\},$ and the maximum $r = r_{\max}$ when $y_2 = \Lambda_3.$  
            Fix $\theta \in \R / 2 \pi \Z$ and consider $u_\theta(s) = (r(s) \cos \theta, r(s) \sin \theta)$ such that  $r(s) $ satisfies
      \[
      \lim_{ s \to \pm \infty } r(s) = 0, \quad r(0) = r_{\max}. 
      \]
      Denote $\Pi( x_1, x_2, y_1, y_2) = (x_1, y_1).$ Then $\Pi^{-1}(u_\theta)$ is an embedded annulus in $\Sigma$ that is transverse to the Hamiltonian flow and asymptotic to $Q_1$ and $Q_2.$ By varying $\theta \in \R/2 \pi \Z,$ we obtain an open book decomposition of $\Sigma$ with annulus-like pages. Each annulus, restricted to the $(x_1, y_1, y_2)$-space, is an arc asymptotic to the points $(0,0,\Lambda_1)$ and $(0,0,\Lambda_{\max})$ that correspond to the binding orbits $Q_1$ and $Q_2,$ respectively.  See Figure \ref{fig:obd}.
    \begin{figure}[ht]
     \centering
\begin{tikzpicture}

 \begin{scope}[scale=0.7 ]
   \draw[thick]   (0,0 ) circle (1cm);
\node at (0,-1.2) [below] {$x_2 \in \mathbb{R} / 2 \pi \mathbb{Z}$};

 \begin{scope}[yshift=-0.5cm, xshift=2.5cm]
\begin{scope}[yshift=-2cm]
   \draw[densely dotted]   (4.5,2.5 ) ellipse (1.5cm and    0.35cm );

  \draw[dashed, ->]  (4.5, 0.5) to (7, 0.5);
    \node at (7.3, 0.5)  {$y_1$};
  \draw[dashed, ->]  (4.5, 0.5) to (2.4, -0.8 );
  \node at (2.4 , -0.8 ) [left] {$x_1$};
 \end{scope}
  \draw[dashed, ->]  (4.5, -1.5) to (4.5, 3);

\draw[gray, thick]   (4.5, 2)   [out=200, in=90]   to  (3.8 , 0.4)  ;
\draw[gray, thick]   (3.8 , 0.4)   [out=270, in=150]   to  (4.5, -1)  ;

\draw[gray, thick]   (4.5, 2)   [out=185, in=90]   to  (3.5 , 0.4)  ;
\draw[gray, thick]   (3.5, 0.4)   [out=270, in=170]   to  (4.5, -1)  ;

\draw[gray, thick]   (4.5, 2)   [out=182, in=90]   to  (3.2, 0.4)  ;
\draw[gray, thick]   (3.2, 0.4)   [out=270, in=175]   to  (4.5, -1)  ;

\draw[gray, thick]   (4.5, 2)   [out=230, in=90]   to  (4.05  , 0.4)  ;
\draw[gray, thick]   (4.05 , 0.4)   [out=270, in=130]   to  (4.5, -1)  ;

\draw[gray, thick]   (4.5, 2)   [out=250, in=90]   to  (4.3  , 0.4)  ;
\draw[gray, thick]   (4.3 , 0.4)   [out=270, in=110]   to  (4.5, -1)  ;


\begin{scope}[xscale=-1, xshift=-9cm]

\draw[gray, thick]   (4.5, 2)   [out=200, in=90]   to  (3.8 , 0.4)  ;
\draw[gray, thick]   (3.8 , 0.4)   [out=270, in=150]   to  (4.5, -1)  ;

\draw[gray, thick]   (4.5, 2)   [out=185, in=90]   to  (3.5 , 0.4)  ;
\draw[gray, thick]   (3.5, 0.4)   [out=270, in=170]   to  (4.5, -1)  ;

\draw[gray, thick]   (4.5, 2)   [out=182, in=90]   to  (3.2, 0.4)  ;
\draw[gray, thick]   (3.2, 0.4)   [out=270, in=175]   to  (4.5, -1)  ;

\draw[gray, thick]   (4.5, 2)   [out=230, in=90]   to  (4.05  , 0.4)  ;
\draw[gray, thick]   (4.05 , 0.4)   [out=270, in=130]   to  (4.5, -1)  ;

\draw[gray, thick]   (4.5, 2)   [out=250, in=90]   to  (4.3  , 0.4)  ;
\draw[gray, thick]   (4.3 , 0.4)   [out=270, in=110]   to  (4.5, -1)  ;

\end{scope}

 \draw[thick] (4.5,0.5) circle (1.5cm);
 \node at (4.5,3.3)  {$y_2$};
 \filldraw[draw=black, fill=black] (4.5, 2) circle (0.07cm);
 \filldraw[draw=black, fill=black] (4.5, -1) circle (0.07cm);
 \draw [->] ( 5.5, 2.5)  [out=190, in=30]  to      (4.55,  2.05);
 \draw [->] ( 5.5, -1 )  [out=200, in=300]  to      (4.55,  -1.05);
 \node at (5.5, 2.5) [right] {$(0,0,\Lambda_{\max})$};
 \node at (5.5, -1 ) [right] {$(0,0,\Lambda_1)$};
  
 \end{scope}
  \end{scope}
     \end{tikzpicture}
    \caption{An open book decomposition of $\Sigma$ with annulus-like pages }
    \label{fig:obd}
 \end{figure}
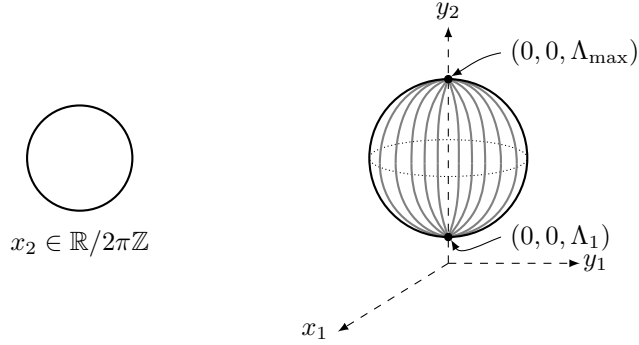

In order to construct a finite energy foliation of $\hat \Sigma$ projecting to the above-described open book decomposition,   choose  $J$ as in Section \ref{sec:lowerconstruction}. We fix $\theta \in \R / 2 \pi \Z$ and  let   $\tilde w_\theta= (d_\theta,w_\theta) \colon \R \times S^1 \to \R \times  \Sigma$ be a $\tilde J$-holomorphic curve. We make the following Ansatz:
 
 \smallskip
 
 \begin{quote} {\bf Ansatz.} The $\Sigma$-part $w_\theta$ is of  the form 
      \begin{equation*}\label{eqww}
      w_\theta(s,t) = ( r(s) \cos \theta  , 2 \pi t, r(s) \sin \theta, y_2(s)),
      \end{equation*}
      where $  r(s)$ satisfies $\lim_{s \to \pm \infty}r(s)=0.$
      \end{quote}
      
      \smallskip

  \noindent    
       In view of $\hat{H} = c,$ we find
       \begin{equation}\label{relationbetweenrandy2}
       \frac{1}{2}r(s)^2 +\hat{H}_2(y_2(s))=c.
       \end{equation}
      
       Proceeding as before, we find  
      \begin{equation}\label{eq:ofy2}
     \dot  y_2 (s) = \frac{ 2 \pi r(s)^2   }{ r(s)^2+ \hat{H}_2'(y_2(s)) ^2} =    \frac{ 4 \pi  (c- \hat{H}_2(y_2(s)) )  }{ 2(c-\hat{H}_2(y_2(s)))+ \hat{H}_2'(y_2(s)) ^2}  ,
      \end{equation}
which might be seen as a differential equation of type
\[
\dot y_2 (s) = P(y_2(s)).
\]
Here, $P=P(y_2)$ is a smooth function   on   $[ \Lambda_1, \Lambda_{\max}],$ which   vanishes precisely at $y_2 = \Lambda_1, \Lambda_{\max}.$ Since $P(y_2)>0, \forall   y_2 \in (\Lambda_1,\Lambda_{\max}),$ we conclude that a solution $y_2=y_2(s)$ to ODE \eqref{eq:ofy2} with initial condition $y_2(0) \in (\Lambda_1, \Lambda_{\max})$ is strictly increasing   and satisfies
            \[
            \lim_{s \to -\infty} y_2(s) = \Lambda_1, \quad \lim_{s \to+\infty} y_2(s) = \Lambda_{\max}.
            \]

We now find using  \eqref{relationbetweenrandy2} and \eqref{eq:ofy2} that 
           \begin{equation*}\label{eqofr}
      \dot r (s) = - \frac{2 \pi r(s)\hat{H}_2'( y_2(s))}{ r(s)^2+ \hat{H}_2'(y_2(s)) ^2}   =   \begin{cases}
     \sqrt{  ( 2 \pi - \dot  y_2 (s)) \dot y_2 (s)} & {\rm if } \; \hat{H}_2'(y_2(s))  \leq 0,\\
  - \sqrt{  ( 2 \pi -\dot  y_2 (s)) \dot y_2 (s)} & {\rm if } \;  \hat{H}_2'(y_2(s))  >0.       \end{cases} 
      \end{equation*}
      In view of \eqref{eq:hatH},
      we conclude that $r=r(s)$ is increasing for $s \in (-\infty, s_0)$ and decreasing for $s \in (s_0, +\infty),$ where   $s_0$ denotes the time at which $y_2(s_0) =\Lambda_3.$ In particular, $r(s) \to 0$ as $s \to \pm \infty.$ 

      We also find
 \[
  ( d_\theta)_s = \hat{\lambda}( (w_\theta)_t)= 2 \pi ( y_2 - \Lambda_3) , \quad (   d_\theta)_t = \hat{\lambda}( (w_\theta)_s) = 0.
 \]
 This implies that $d$ is independent of $t$, and hence  we may write
 \begin{equation*}\label{eqd}
 d_\theta(s,t) = 2 \pi \int_0^s  ( y_2( u ) - \Lambda_3) d u.
 \end{equation*}
Note that $d_\theta(s,t) \to +\infty$ as $s \to \pm \infty.$

We have constructed a finite energy $\tilde J$-holomorphic cylinder $\tilde w_\theta = (d_\theta, w_\theta)\colon \R \times S^1 \to  \R \times \Sigma$ that is asymptotic to $Q_1$ and $Q_2$ at $\mp \infty,$ respectively. Note that both asymptotic limits are positive. Its Hofer energy equals the sum of the Reeb periods of $Q_1$ and $Q_2.$

Since the above argument is independent of the choice of $\theta \in \R / 2 \pi \Z,$ by varying $\theta$ we obtain a finite energy foliation of $\R \times \hat \Sigma,$ which projects 
  the open book decomposition of $\hat \Sigma$ with annulus-like pages, decribed above. 
  Its intersection with the set $\Sigma_{\rm direct}^u$ yields the desired annulus foliation. 
Namely,   the solid torus $\Sigma_{\rm direct}^u$ is foliated into annuli transverse to the Hamiltonian flow restricted to $\Sigma_{\rm direct}^u.$  The   outer boundary of each annulus    lies  on the boundary torus  $\partial \Sigma_{\rm direct}^u.$ The inner boundary coincides with the direct circular orbit  $\gamma_{\rm direct}^u.$ 
  
    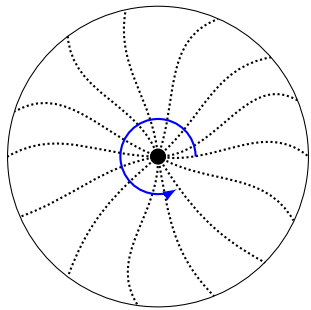
\begin{figure}[ht]
     \centering
\begin{tikzpicture} 
 
  \filldraw[draw=black,fill=white] (0,0 ) circle (2cm);
  
   \draw[thick, densely dotted] (0,0) [ out=350, in=140] to (2,0);
   \draw[thick, densely dotted] (0,0) [ out=0, in=150]  to (1.85 ,0.75);
   
   \draw[thick, densely dotted] (0,0) [out=40, in=190] to (1.5, 1.32);

   \draw[thick, densely dotted] (0,0)  [ out=70, in=200] to (0.8,1.83);
         \draw[thick, densely dotted] (0,0)  [ out=95, in=260] to (-0.43 ,1.95);
         
         \draw[thick, densely dotted] (0,0) [out=110, in=275] to (-1.2, 1.6);
         
   \draw[thick, densely dotted] (0,0) [ out=160, in=30]  to (-1.9,0.6);

            \draw[thick, densely dotted] (0,0) [out=180, in=30] to  (-2,0);

   \draw[thick, densely dotted] (0,0)  [ out=190, in=20] to (-1.83,-0.8);
      \begin{scope}[rotate=-10]
   \draw[thick, densely dotted] (0,0)  [ out=230, in=80] to (-0.9,-1.77 );
      \end{scope}
      \begin{scope}[rotate=-18]
   \draw[thick, densely dotted] (0,0)  [ out=290, in=140] to (0.35,-1.96);
         \end{scope}
         \begin{scope}[rotate=-15]
   \draw[thick, densely dotted] (0,0)  [ out=295, in=150] to (1.2,-1.6);
                  \end{scope}
                  \begin{scope}[rotate=-8]
   \draw[thick, densely dotted] (0,0)  [ out=310, in=160] to (1.61,-1.2);
\end{scope}
   \draw[thick, densely dotted] (0,0)  [ out=330, in=120] to (1.8,-0.85);
 
  \draw[thick, blue, ->]   (0.5,0) arc (0:300:0.5cm and 0.5cm);

   \filldraw[draw=black ,fill=black] (0,0 ) circle (0.1cm);

  \end{tikzpicture}
    \caption{A slice of the annulus foliation. The dot and the dotted curves correspond  to the direct circular orbit and  annuli, respectively. The blue curve points in the direction of the Hamiltonian flow. }
    \label{fig:slice}
 \end{figure}

  \bibliographystyle{abbrv}
\bibliography{mybibfile}

\end{document}